\documentclass{amsart}
\usepackage{amsmath, amssymb,epic,graphicx,mathrsfs,enumerate}
\usepackage[all]{xy}

\usepackage[utf8]{inputenc}

\usepackage{geometry}

\usepackage{amsthm}
\usepackage{amssymb}
\usepackage{latexsym}
\usepackage{longtable}
\usepackage{epsfig}
\usepackage{amsmath}
\usepackage{hhline}
\usepackage{amscd}
\usepackage{newlfont}

\usepackage{color}

\DeclareMathOperator{\Inndiag}{Inndiag}
\DeclareMathOperator{\Sym}{Sym}

\DeclareMathOperator{\Ord}{Ord}

\DeclareMathOperator{\Frat}{Frat}

\DeclareMathOperator{\SL}{SL}
\DeclareMathOperator{\Ug}{U}
\DeclareMathOperator{\GL}{GL}
\DeclareMathOperator{\Sp}{Sp}

\DeclareMathOperator{\PSp}{PSp}
\DeclareMathOperator{\POmega}{P\Omega}
\DeclareMathOperator{\Stab}{Stab}

\DeclareMathOperator{\Alt}{Alt}
\DeclareMathOperator{\Aut}{Aut}

\DeclareMathOperator{\Lg}{L}
\DeclareMathOperator{\Out}{Out}

\DeclareMathOperator{\Syl}{Syl}

 \newcommand{\ol}{\overline}

\newtheorem{thma}{Theorem}
\newtheorem{cora}[thma]{Corollary}
\newtheorem{thm}{Theorem}[section]
\newtheorem{cor}[thm]{Corollary}

 \newtheorem{lemma}[thm]{Lemma}
\newtheorem{prop}[thm]{Proposition}

\numberwithin{equation}{section}

\begin{document}
\bibliographystyle{amsplain}
\title[Proportions of $p$-elements]{On the proportion of $p$-elements in a finite group, and a modular Jordan type theorem\thanks{Supported by the Engineering and Physical Sciences Research Council, grant number EP/T017619/1.}}
\author{Gareth Tracey}
\address{School of Mathematics,
University of Birmingham,
Edgbaston,
Birmingham,
B15 2TT\\
United Kingdom}
\email{g.tracey@bham.ac.uk}


\begin{abstract}
In 1878, Jordan proved that if a finite group $G$ has a faithful representation of dimension $n$ over $\mathbb{C}$, then $G$ has a normal abelian subgroup with index bounded above by a function of $n$. The same result fails if one replaces $\mathbb{C}$ by a field of positive characteristic, due to the presence of large unipotent and/or Lie type subgroups. For this reason, a long-standing problem in group and representation theory has been to find the ``correct analogue" of Jordan's theorem in characteristic $p>0$. Progress has been made in a number of different directions, most notably by Brauer and Feit in 1966; by Collins in 2008; and by Larsen and Pink in 2011. With a 1968 theorem of Steinberg in mind (which shows that a significant proportion of elements in a simple group of Lie type are unipotent), we prove in this paper that if a finite group $G$ has a faithful representation over a field of characteristic $p$, then a significant proportion of the elements of $G$ must have $p$-power order. We prove similar results for permutation groups, and present a general method for counting $p$-elements in finite groups. All of our results are best possible.
\end{abstract}

\maketitle
\section{Introduction}
The study of the dimensions of the faithful representations of a finite group $G$ (over an arbitrary field) has a long and rich history. Amongst the early results in this direction, one of the most famous is due to Jordan \cite{Jordan}, who proved in 1878 that if $n$ is the dimension of a faithful $\mathbb{C}[G]$-representation, then $G$ has an abelian normal subgroup $A$ whose index is bounded above by a function $J(n)$ of $n$. The function $J(n)$ was shown to be at most $(n+1)!$ (for large enough $n$) by Collins \cite{CollinsC} in 2007 (this bound is best possible -- realised by the symmetric group acting on the fully deleted permutation module over $\mathbb{C}$). 

If we replace $\mathbb{C}$ by a field $\mathbb{F}$ of characteristic $p>0$, then Jordan's theorem fails, even if $\mathbb{F}$ is algebraically closed. For example, the quasisimple group $\SL_2(\mathbb{F}_{p^n})$ embeds into $\GL_{2n}(\mathbb{F}_p)\le \GL_{2n}(\ol{\mathbb{F}_p})$, and the index of the largest abelian normal subgroup of $\SL_2(\mathbb{F}_{p^n})$ depends on the prime $p$. 

For this reason, finding the ``correct" modular analogue of Jordan's theorem has been a well-studied problem in representation theory for a number of years. In one direction, Brauer and Feit \cite{BF} (in proving a conjecture of Kegel) showed in 1966 that if $\mathbb{F}$ is a field of characteristic $p>0$ and $G$ is a finite subgroup of $\mathrm{GL}_n(\mathbb{F})$ with a Sylow $p$-subgroup of order $p^m$, then $G$ has an abelian normal subgroup of index bounded above by a function of $m$ and $n$. In another direction, Collins \cite{Collinsp} relaxed the requirement that the normal subgroup $A$ of $G$ be abelian, and proved that if $\mathbb{F}$ and $G\le \mathrm{GL}_n(\mathbb{F})$ are as above, with $O_p(G)=1$, then $G$ has an abelian normal subgroup $A$ with $E_p(G)A$ having index bounded above by a function of $n$ (the group $E_p(G)$ is the subgroup of $G$ generated by the quasisimple subnormal subgroups of $G$ which are Lie type groups in characteristic $p$). A remarkable paper by Larsen and Pink \cite{LP} later proved a similar result, but without relying on the Classification of Finite Simple Groups (henceforth abbreviated to CFSG). See the preprint \cite{Breuillard} by Breuillard for another approach to the problem.

As the discussion above alludes to, the obstacle to Jordan's theorem in characteristic $p>0$ is the presence of large Lie type subgroups in characteristic $p$, and large unipotent subgroups.
By a 1968 theorem of Steinberg \cite{S}, one property that is shared by these two classes of linear groups is that a ``large" proportion of their elements have $p$-power order. Of course, all elements in a finite unipotent linear group have $p$-power order; Steinberg's theorem shows that around $1/q^{f(n)}$ elements of a Lie type group $X=X_n(q)$ in characteristic $p$ have $p$-power order, where $f(n)$ is a linear function of the Lie rank $n$ of $X$ (we discuss this in more detail below).

In this paper, we prove a new modular analogue of Jordan's theorem by showing that although it is not true that a significant number of elements in a finite linear group in characteristic $p>0$ are contained in an abelian normal subgroup, it is true that a significant number of elements have $p$-power order. This also shows that Steinberg's result is indicative of the general case for finite linear groups in characteristic $p$.

To state our theorem, we fix the following notation: for a positive real number $m$, define $h_p(m):=3^{(m-1)/2}$ if $p=2$; $h_p(m):=20^{(m-1)/4}$ if $p=3$; and $h_p(m):=(p-1)!^{(m-1)/(p-2)}$ if $p>3$. Our theorem can now be stated as follows. 
\begin{thma}\label{thm:JordanAnalogue}
Let $G$ be a finite group, let $p$ be prime, and suppose that $n$ is the dimension of a faithful representation of $G$ over a field $\mathbb{F}_q$ of size $q=p^f$. Define $\alpha(q):=96$ if $q:=5$, $\alpha(q):=144$ if $q=7$, and $\alpha(q):=600$ if $q=11$. Then the proportion of elements of $G$ of $p$-power order is at least $1/f(n,q)$, where
\[f(n,q):= \begin{cases}
      [3(q^3-1)]^{n/3}h_p(n/3) & \text{if $p=2$;} \\
      640^{n/4}h_p(n/4) & \text{if $q=3$;} \\
      [2(q^2-1)]^{n/2}h_p(n/2) & \text{if $q\in\{9,27\}$;}\\
      \alpha(q)^{n/2}h_p(n/2) & \text{if $q\in\{5,7,11\}$;} \\
      (q-1)^{n}h_p(n) & \text{otherwise.}
   \end{cases}
\]
In particular, $|G:O^{p'}(G)|\le f(n,q)$.
\end{thma}

\noindent We remark that constructing elements of prime power order in a given finite group $G$ is an important tool in computational group theory (see for example \cite{Kantor,KSComp,KSLie}). In practice, this is done by choosing elements at random (with an ``almost" uniform distribution) and computing their orders. Theorem \ref{thm:JordanAnalogue} shows that if $G$ is constructed as a linear group and $p$ is the characteristic of the underlying field, then one should find a $p$-element relatively quickly in this way.

Our methods also allow us to prove the following slightly stronger result, when the representation in question is irreducible.
\begin{thma}\label{thm:JordanIrr}
Let $G$ be a finite group, let $p$ be prime, and suppose that $n$ is the dimension of a faithful irreducible representation of $G$ over a field $\mathbb{F}_q$ of size $q=p^f$. Define $\alpha(q):=96$ if $q:=5$, $\alpha(q):=144$ if $q=7$, and $\alpha(q):=600$ if $q=11$. Then the proportion of elements of $G$ of $p$-power order is at least $1/i(n,q)$, where
\[i(n,q):= \begin{cases}
      [3(q^3-1)]^{n/3}h_p(n/3) & \text{if $p=2$;} \\
      640^{n/4}h_p(n/4) & \text{if $q=3$;} \\
      [2(q^2-1)]^{n/2}h_p(n/2) & \text{if $q=3$ and $4\nmid n$;} \\
      [2(q^2-1)]^{n/2}h_p(n/2) & \text{if $q\in\{9,27\}$;}\\
      \alpha(q)^{n/2}h_p(n/2) & \text{if $q\in\{5,7,11\}$;} \\
      (q-1)^{n}h_p(n) & \text{if $q\in\{5,7,11\}$ and $n$ is odd;} \\
      (q-1)^{n}h_p(n) & \text{otherwise.}
   \end{cases}
\]
In particular, $|G:O^{p'}(G)|\le i(n,q)\le f(n,q)$, where $f(n,q)$ is as in Theorem \ref{cor:JordanAnalogue}.
\end{thma}

\noindent We remark that it is not immediately obvious that the function $i(n,q)$ defined in Theorem \ref{thm:JordanIrr} is bounded above by the function $f(n,q)$ from Theorem \ref{thm:JordanAnalogue}. This is however, easy to see, and will be proved in Lemma \ref{lem:mtbound}.

The bounds in Theorems \ref{thm:JordanAnalogue} and \ref{thm:JordanIrr} are best possible. To see this, define the integer $m$ and the permutation group $S$ as follows: if $p=2$, then set $m=3^{k}$ and let $S$ be a Sylow $3$-subgroup of $\Sym_{m}$. If $p=3$, then set $m=5^{k}$, and let $S:=(C_5\wr C_4)\wr\hdots\wr(C_5\wr C_4)\le \Sym_m$ be the iterated wreath product of $k$ copies of $C_5\rtimes C_4\le \Sym_5$. If $p\geq 5$, then set $m:=(p-1)^k$, and let $S:=\Sym_{p-1}\wr\hdots\wr\Sym_{p-1}\le \Sym_m$ be the iterated wreath product of $k$ copies of $\Sym_{p-1}$.
Also, define the integer $r$ and the irreducible linear group $R\le \GL_r(q)$ as follows: if $q=2$, then set $r:=3$ and define $R$ to be the semilinear group $R:=\Gamma \mathrm{L}_1(q^r)\cong (q^r-1)\rtimes r\le \GL_r(q)$. If $q=3$ then set $r:=4$ and $R:=2^{1+4}.(C_5\rtimes C_4)\le \GL_4(q)$. If $q\in\{9,27\}$, then set $r:=2$ and $R:=\Gamma \mathrm{L}_1(q^r)$. If $q\in\{5,7\}$, then set $r:=4$ and $R:=(q-1)\circ 2^{1+2}.\Sp_2(2)\le \GL_4(q)$. If $q=11$, then set $r:=2$ and $R:=2.\Alt_5\le \GL_2(11)$. In all other cases, define $r:=1$ and $R:=\GL_1(q)$.

Finally, set $n:=rm$ and $G:=R\wr S\le \GL_n(q)$. In each case, $G$ is an irreducible $p'$-group of order $i(n,q)\le f(n,q)$. It also follows from these examples that the bounds in Theorems \ref{thm:JordanAnalogue} and \ref{thm:JordanIrr} are best possible even if one replaces ``$p$-elements" by ``$p$-singular elements" in the statements of the theorems. (An element $g$ of a finite group $G$ is said to be \emph{$p$-singular} if $p$ divides the order of $g$.)
\footnote{for bounds in Theorems \ref{thm:JordanAnalogue} and \ref{thm:JordanIrr} that are exponential in $q$, note that $(p-1)!^{1/(p-2)}\le p-1$.}

The bounds in Theorems \ref{thm:JordanAnalogue} and \ref{thm:JordanIrr} essentially follow from two other results, one concerning unipotent elements in primitive linear groups, and the other concerning $p$-elements in permutation groups. 

We will now state these results, beginning with the former. Recall that an irreducible linear group $X\le \mathrm{GL}(V)$ is said to be \emph{imprimitive} if the natural $X$-module $V$ has a direct sum decomposition $V=V_1\oplus\hdots\oplus V_t$ with the property that for each $1\le i\le t$, and each $x\in X$, either $V_i^x=V_i$ or $V_i^x\cap V_i=0$. If no such decomposition exists, then the irreducible group $X\le \mathrm{GL}(V)$ is said to be \emph{primitive}. Our lower bound on the proportion of unipotent elements in a primitive linear group can now be given as follows. 
\begin{thma}\label{thm:MainTheoremPrimLin}
Let $G$ be a finite group, let $p$ be prime, and suppose that $n$ is the dimension of a faithful primitive representation of $G$ over a field $\mathbb{F}_q$ of size $q=p^f$. 
\begin{enumerate}[\upshape(i)]
    \item If $(G,n,q)$ lies in Table \ref{tab:EXMainTheoremPrimLin}, then the proportion of elements of $G$ of $p$-power order is precisely $1/P(G,n,q)$, where $P(G,n,q)$ is as given in the fourth column of the table.
    \item If $(G,n,q)$ does not lie in Table \ref{tab:EXMainTheoremPrimLin}, then the proportion of elements of $G$ of $p$-power order is at least $1/[n_{p'}(q^n-1)]$.
\end{enumerate}
\end{thma}
\noindent Here, and throughout the paper, $n_{p'}$ is defined to be the the $p'$-part of $n$. 

\begin{table}[]
    \centering
    \begin{tabular}{c|c|c|c}
     $G$    & $n$ & $q$ & $P(G,n,q)$\\
     \hline
     $6\circ 3^{1+2}.\Sp_2(3)$ & $3$    & $7$ & $1296$\\
     $(q-1)\circ 2^{1+2}.\Sp_2(2)$ & $2$    & $5,7$ & $96,144$\\
     $2^{1+4}.\Sp_4(2)$ & $4$    & $7$ & $23040$\\
     $(q-1)\circ 2^{1+4}.\Sp_4(2)$ & $4$    & $7,11,13$ & $69120,115200,138240$\\
     $2^{1+4}.(5\rtimes 4)$ & $4$    & $3$ & $640$\\
     $4\circ 2^4.(\Sp_2(2)\wr \Sym_2)$ & $4$    & $5$ & $4608$\\ 
     $2^{1+4}.\Sp_4(2)'$ & $4$    & $7$ & $11520$\\ 
     $6\circ 2^{1+4}.\Sp_4(2)'$ & $4$    & $7$ & $34560$\\ 
     $6\circ 2^{1+4}.\Sym_5$ & $4$    & $7$ & $11520$\\
     $(q-1)\circ 2.\Alt_5$ & $2$    & $11,19$ & $600,1140$\\
     $10\circ (2.\Alt_5\circ 2.\Alt_5).2$ & $4$    & $11$ & $72000$\\
     \hline
    \end{tabular}
    \caption{The triples ($G,n,q)$ from Theorem \ref{thm:MainTheoremPrimLin}(i).}
    \label{tab:EXMainTheoremPrimLin}
\end{table}

The lower bound in Theorem \ref{thm:MainTheoremPrimLin}(ii) is best possible, as one can see by considering a Hall $p'$-subgroup of the semilinear group $\Gamma\mathrm{L}_1(q^n)\cong (q^n-1)\rtimes n\le \GL_n(q)$. We remark that the bound in Theorem \ref{thm:MainTheoremPrimLin} is also best possible (up to multiplication by a ``small" constant) when one considers primitive linear groups of order divisible by $p$. For example, in the theorem of Steinberg \cite{S} mentioned above, it is proved that the number of elements of $p$-power order in $\mathrm{\GL}_n(q)$ is precisely $q^{n(n-1)}$. It follows that the proportion of elements of $p$-power order in $\mathrm{\GL}_n(q)$ is  
\begin{align*}
    \dfrac{q^{n(n-1)}}{|\GL_n(q)|}&=\dfrac{q^{n(n-1)/2}}{\prod_{i=1}^n(q^i-1)}=\prod_{i=1}^n\left(\dfrac{q^i}{q^i-1}\right) \dfrac{q^{n(n-1)/2}}{\prod_{i=1}^nq^i}=\prod_{i=1}^n\left(\dfrac{q^i}{q^i-1}\right)\dfrac{1}{q^n}.
\end{align*}
Since $c:=\prod_{i=1}^n\left(\dfrac{q^i}{q^i-1}\right)$ is at least $1$, and $1/c=\prod_{i=1}^n\left(1-\dfrac{1}{q^i}\right)\geq 1-1/q\geq 1/2$, we see that the proportion of elements of $\GL_n(q)$ of $p$-power order lies between $1/q^n$ and $2/q^n$. (We refer the reader to \cite[Lemma 3.1]{LMP} for tighter upper and lower bounds on the constant $c$).

Before proceeding to our results on permutation groups, we state a useful consequence of Theorems \ref{thm:JordanAnalogue}, \ref{thm:JordanIrr}, and \ref{thm:MainTheoremPrimLin}. For a prime $p$, let $e=\mathrm{exp}_{p'}(G):=\mathrm{lcm}\{|x|\text{ : }x\in G\text{, }(|x|,p)=1\}$. Since all modular representations of a finite group $G$ in characteristic $p$ can be realised over the field $\mathbb{F}_{p^e}$, we have: 
\begin{cora}\label{cor:JordanAnalogue}
Let $\mathbb{F}$ be a field of characteristic $p>0$, let $G$ be a finite group, and suppose that $n$ is the dimension of a faithful representation $\rho$ of $G$ over $\mathbb{F}$. Set $e:=\mathrm{exp}_{p'}(G)$. Then:
\begin{enumerate}[\upshape(i)]
    \item The proportion of elements of $G$ of $p$-power order is at least $1/f(n,p^e)$, where $f(n,p^e)$ is as defined in Theorem \ref{thm:JordanAnalogue}. In particular, $|G:O^{p'}(G)|\le f(n,p^{e})$.
    \item If $\rho$ is irreducible, then the proportion of elements of $G$ of $p$-power order is at least $1/i(n,p^e)$, where $i(n,p^e)$ is as defined in Theorem \ref{thm:JordanAnalogue}. In particular, $|G:O^{p'}(G)|\le i(n,p^{e})$.
    \item If $\rho$ is primitive and $\rho(G)$ is not one of the groups in Table \ref{tab:EXMainTheoremPrimLin}, then the proportion of elements of $G$ of $p$-power order is at least $1/[n_{p'}(p^{en}-1)]$. In particular, $|G:O^{p'}(G)|\le n_{p'}(p^{en}-1)$.
\end{enumerate}
\end{cora}

Next, we state our main result for permutation groups.
\begin{thma}\label{thm:MainTheoremPerm}
Let $G\le\Sym_n$ be a permutation group, and let $p$ be prime. Then the proportion of elements of $G$ of $p$-power order is at least $1/h_p(n)$, where $h_p(n)$ is as defined on page 2.
\end{thma}
The bound in Theorem \ref{thm:MainTheoremPerm} is again best possible: one can see this by taking $n:=m$ and $G:=S\le\Sym_n$, where $m$ and $S$ are as defined in the discussion after the statement of Theorem \ref{thm:JordanIrr}.

Write $\mathrm{Lie}(p)$ for the set of finite simple groups of Lie type in characteristic $p$. As a by-product of our methods we obtain upper bounds on the minimal centralizer orders of non-trivial $p$-elements in the $\Aut(S)$-cosets of $S$. In order to state our result, we require some standard notation and terminology: If $S\in\mathrm{Lie}(p)$, then $S=O^{p'}(\ol{X}_{\ol{\sigma}})$, where $\ol{X}$ is a simple algebraic group over the algebraic closure $\ol{\mathbb{F}_p}$ of $\mathbb{F}_p$, $\ol{\sigma}:\ol{X}\rightarrow \ol{X}$ is a Steinberg endomorphism of $\ol{X}$, and $\ol{X}_{\ol{\sigma}}=\{x\in \ol{X}\text{ : }\ol{\sigma}(x)=x\}$. In particular, by \cite[Theorem 2.2.3]{GLS}, we may write $\ol{\sigma}=\ol{\sigma_p}^f\ol{\gamma}^i$, where $\ol{\sigma_p}$ is a Frobenius automorphism of $\ol{X}$, and either 
\begin{enumerate}[(a)]
    \item $\ol{X}\in\{A_{\ell},D_{\ell},E_6\}$ and $\ol{\gamma}$ is a graph automorphism of $\ol{X}$ as in \cite[Theorem 1.15.2(a)]{GLS};
    \item $\ol{X}\in\{B_2\text{ }(p=2),G_2\text{ }(p=3)\text{, }F_4\text{ }(p=2)\}\}$, and $\ol{\gamma}$ is as in \cite[Theorem 1.15.4(b)]{GLS}; or
    \item $\ol{\gamma}=1$.
\end{enumerate}
In this case, we define $p^f$ to be the \emph{level} of $S$.

Our result for centralizer orders of $p$-elements in $\Aut(S)$-cosets of $S$ can now be stated as follows.
\begin{thma}\label{thm:MainTheoremp}
Fix a prime $p$. Let $S$ be a group in $\mathrm{Lie}(p)$ of level $q$ and untwisted Lie rank $\ell$, and write $S=\ol{X}_{\ol{\sigma}}$, for a simple algebraic group $\ol{X}$ and a Steinberg endomorphism $\ol{\sigma}:\ol{X}\rightarrow\ol{X}$, as above. Let $\alpha$ be a $p$-element of $\Aut(S)$. Then there exists an element $x$ of the coset $S\alpha$ such that one of the following holds:
\begin{enumerate}[\upshape(i)]
    \item $|C_{\ol{X}_{\ol{\sigma}}}(x)|\le q^{\ell}$; or 
    \item $S$ and $p$ are as in Table \ref{tab:MainTheorempTable}. In this case, an upper bound for $|C_{\ol{X}_{\ol{\sigma}}}(x)|$ is given in the third column of Table \ref{tab:MainTheorempTable}.
\end{enumerate}
\end{thma}
\noindent We will see in Section \ref{sec:prproofs} that when $\alpha=1$, the element $x$ chosen in Theorem \ref{thm:MainTheoremp} is a so-called \emph{regular unipotent} element of $S$. Moreover, the primes involved in Theorem \ref{thm:MainTheoremp}(ii) are the \emph{bad primes} associated to the simple algebraic group $\ol{X}$ (see \cite[page 28]{Carter} for more information).
   \begin{table}[ht!]
       \centering
       \begin{tabular}{c|c|c}
        $S$    & $p$ &  $|C_{\ol{X}_{\ol{\sigma}}}(x)|\le $\\
        \hline
        $\PSp_{2\ell}(q)$    & $2$  &  $2q^{\ell}$\\
        $\POmega^{\pm}_{2\ell}(q)$    & $2$  &  $2q^{\ell}$\\
        ${}^2B_2(q)$    & $2$  &  $2q$ \\
        ${}^2G_2(q)$    & $3$  & $3q$ \\
        ${}^3D_4(q)$    & $2$  & $2q^4$ \\
        $G_2(q)$        &  $2$ & $2q^2$  \\
                        &  $3$ & $3q^2$\\
        ${}^2F_4(q)$    & $2$  & $4q^4$\\
        $F_4(q)$        &  $2$ & $4q^4$\\
                        &  $3$ & $3q^4$\\
        $E^{\pm}_6(q)$        &  $2,3$ & $6q^6$\\
        $E_7(q)$        &  $2$ &  $4q^7$ \\
                        &  $3$ &  $6q^7$\\
        $E_8(q)$        &  $2$ &   $4q^8$\\
                        &  $3$ &   $3q^8$\\ 
                        &  $5$ &   $5q^8$\\ 
                        \hline
       \end{tabular}
       \caption{Upper bounds for centralizer orders of regular unipotent elements of certain simple groups.}
       \label{tab:MainTheorempTable}
   \end{table}

We also prove the following order bound for sporadic or cross characteristic Lie type sections of linear groups.
\begin{thma}\label{thm:MainTheoremr}
Let $p$ be a prime, and let $S$ be a nonabelian finite simple group which is either sporadic or lies in $\mathrm{Lie}(r)$. Suppose that $S$ is a section of $\mathrm{GL}_n(\mathbb{F})$ for some field $\mathbb{F}$ of characteristic $p>0$, with $p\neq r$. Assume also that $S$ is not one of the groups
$M_{12},M_{22},M_{24},J_2,\mathrm{Suz},Co_1,Co_2,\Lg_2(7),\Lg_2(8),\Lg_2(9),\Lg_3(4),\Ug_4(2),\Ug_4(3),\PSp_6(2),$ \\ $\PSp_6(3)$, or $\POmega^+_8(2)$.
Then $|S|\le 2^{2(\log{n})^2+\log_2{n}}$.
\end{thma}

The bound in Theorem \ref{thm:MainTheoremr} is best possible, up to multiplication by an absolute constant. Indeed, as we will see in Section \ref{sec:asymptotic}, there are primitive linear groups $G\le\mathrm{GL}_{2^m}(\ol{\mathbb{F}_p})$, for $p$ odd, with sections isomorphic to $\Sp_{2m}(2)$ (and $\Sp_{2m}(2)$ has order $c2^{2m^2+m}$, where $1/2< c< 1$ is an absolute constant).

The layout of the paper is as follows. In Section \ref{sec:pcount}, we prove two key inductive lemmas which will enable us to obtain lower bounds on the number of $p$-elements in any given finite group. The second of these lemmas will also demonstrate the motivation behind Theorems \ref{thm:MainTheoremp} and \ref{thm:MainTheoremr}. In Section \ref{sec:AltThmSec}, motivated by our work in Section \ref{sec:pcount}, we derive an upper bound on the minimal order of a centralizer of a $p$-element in a finite alternating group. We then have all the tools we need to prove Theorem \ref{thm:MainTheoremp}, and we do so in Section \ref{sec:prproofs}. In Section \ref{sec:primsec}, we make a necessary detour to prove some results on the structure of a finite primitive linear group. These results will allow us to prove Theorems \ref{thm:MainTheoremPerm} and \ref{thm:MainTheoremr}, and we do this in Section \ref{sec:asymptotic}. Finally, we prove Theorem \ref{thm:MainTheoremPrimLin} in Section \ref{sec:PrimLin}, and Theorems \ref{thm:JordanAnalogue} and \ref{thm:JordanIrr} in Section \ref{sec:Jordanproof}.

\vspace{0.25cm}
\noindent \textbf{Notation}: Most of our notation is standard. For a prime $r$, we write $r^m$ and $r^{1+2m}$ for elementary abelian and extraspecial $r$-groups of order $r^m$ and $r^{1+2m}$, respectively. We also occasionally write $n$ for the cyclic group $\mathbb{Z}/n\mathbb{Z}$, when there is no danger of confusion. For almost simple group names, we follow \cite{KL}, except that we use $\Alt_n$ and $\Sym_n$ for the alternating and symmetric groups of degree $n$. When it is natural to do so, we will also occasionally use Lie theoretic notation for simple groups in place of classical notation (so $A_{\ell}(q)$ instead of $\Lg_{\ell+1}(q)$ etc.). 

For a finite group $G$, we will write $Z(G)$, $\mathrm{Frat}(G)$, $F^*(G)$, and $G'$ for the centre, Frattini subgroup, generalized Fitting subgroup, and derived subgroup of $G$, respectively. For a field $\mathbb{F}$, we will write $\ol{\mathbb{F}}$ for the algebraic closure of $\mathbb{F}$. Finally, for two groups $A$ and $B$, $A\circ B$ will denote the central product of $A$ and $B$.

\section{Counting $p$-elements in finite groups}\label{sec:pcount}
As mentioned above, the purpose of this section is to prove two inductive lemmas which will allow us to count the number of $p$-elements in a given finite group $G$. The first of these lemmas will enable us to reduce our enumeration to the case where $G$ is characteristically simple, while the second gives a lower bound on the number of $p$-elements of $G$ in terms of similar numbers associated to the composition factors of $G$. The first lemma can in fact be proved when $p$ is replaced by an arbitrary finite set of primes, and we do prove this more general version here.

In order to state the lemmas, we require some notation. For a subset $X$ of a finite group $G$ and a set $\pi$ of primes, write $\mathrm{Ord}(X,\pi)$ for the set of $\pi$-elements of $X$ (that is, the set of elements $x\in X$ such that all prime divisors of $|x|$ lie in $\pi$). If $G=S^t$ is the direct power of a simple group $S$, define the function $f_{\pi}$ by $f_{\pi}(G):=1$ if $S$ is an abelian $\pi'$-group; $f_{\pi}(G):=|G|$ if $G$ is an abelian $\pi$-group; and 
\begin{align*}
f_{\pi}(G):=\min\{|\Ord(G\alpha,\pi)|\text{ : }\alpha\in \mathrm{Ord}(\Aut(G),\pi)\}
\end{align*}
if $S$ is nonabelian. In this latter case, we are viewing $G\alpha$ as a subset of $\Aut(G)$. If $\pi$ consists of a single prime $p$, then we will write $\Ord(G,p)$ and $f_p(G)$ in place of $\Ord(G,\{p\})$ and $f_{\{p\}}(G)$.

The first lemma can now be given as follows. 
\begin{lemma}\label{lem:CrucialRed}
Let $G$ be a finite group, and let $\pi$ be a set of primes.
Then $|\Ord(G,\pi)|\geq\prod_{C\text{ a chief factor of }G}f_{\pi}(C)$.
\end{lemma}
\begin{proof}
Let $N$ be a minimal normal subgroup of $G$. Clearly, it will suffice to prove that 
\begin{align}\label{lab:top}
 |\Ord(G,\pi)|\geq f_{\pi}(N)|\Ord(G/N,\pi)|.    
\end{align}
To this end, notice that a (right) coset $Ng$ of $N$ in $G$ contains a $\pi$-element only if $Ng$ is a $\pi$-element of $G/N$. Thus, writing bars to denote reduction modulo $N$, we have $|\Ord(G,\pi)|=\sum_{\ol{g}\in \Ord(\ol{G},\pi)} |\Ord(Ng,\pi)|$.
Thus, to prove (\ref{lab:top}), it will suffice to show that if $\ol{g}\in \Ord(\ol{G},\pi)$, then 
\begin{align}\label{lab:top2}
\text{$|\Ord(Ng,\pi)|\geq f_{\pi}(N)$.}  
\end{align}
So fix $g\in G$ with $\ol{g}\in \Ord(\ol{G},\pi)$. Since $\Ord(\ol{G},\pi)=\Ord(G,\pi)N/N$, we may assume that $g\in \Ord(G,\pi)$. Set $H:=\langle N,g\rangle=N\langle g\rangle$. Suppose first that $N$ is an elementary abelian $r$-group, for some prime $r$. If $r\not\in\pi$, then $f_{\pi}(N)=1$; while if $r\in\pi$, then $H$ is a $\pi$-group, so $|\Ord(Ng,\pi)|=|Ng|=|N|=f_{\pi}(N)$. The assertion (\ref{lab:top2}) follows in each case.

Finally, assume that $N$ is non-abelian.
Since $C_H(N)\cap N=1$ and $\pi(|H/N|)\subseteq\pi$, the group $C_H(N)$ is a $\pi$-group. It follows that $|\Ord(Ng,\pi)|$ is at least the number of $\pi$-elements in the coset $\theta(Ng)$, where $\theta$ is the natural map $\theta:H\rightarrow H/C_H(N)$. Since $|\Ord(\theta(Ng),\pi)|\geq f_{\pi}(N)$ by definition of $f_{\pi}$, the claim at (\ref{lab:top2}) follows, and the proof is complete.
\end{proof}

To apply Lemma \ref{lem:CrucialRed} in the case $\pi=\{p\}$, we will need lower bounds on the quantity $f_{p}(N)$ when the chief factor $N$ of $G$ is nonabelian. Our next lemma shows that such bounds can be derived from information on the $S$-conjugacy classes in $\Aut(S)$-cosets of finite simple groups $S$.
\begin{lemma}\label{lem:fpnonab}
Fix a nonabelian finite simple group $S$, a positive integer $t$, and a prime $p$. Set $N:=S^t$, and let $f_{p}(N)$ be as defined above. Assume that each coset of $S$ in $\Aut(S)$ contains an $S$-conjugacy class of $p$-elements of size at least $M$. Then $f_{p}(N)\geq M^t$.
\end{lemma}
\begin{proof}
We note first that $\Aut(N)\cong \Aut(S)\wr\Sym_t$, and the conjugation action of $\Aut(N)$ on $S^t$ is given by $(s_1,\hdots,s_t)^{(h_1,\hdots,h_t)}=(s_1^{h_1},\hdots,s_t^{h_t})$, and $(s_1,\hdots,s_t)^{\sigma^{-1}}=(s_{1^{\sigma}},\hdots,s_{t^{\sigma}})$, where $h_i,s_i\in S$, and $\sigma\in\Sym_t$.

Now, let $g$ be a $p$-element of $\Aut(N)$, and write $g=(x_1,\hdots,x_t)\sigma$, where $g_i\in \Aut(S)$, and $\sigma\in \Sym_t$. If $P\in\Syl_p(\Aut(S))$ and $Q\in\Syl_p(\Sym_t)$, then $P\wr Q\cong P^t\rtimes Q$ is a Sylow $p$-subgroup of $\Aut(N)$. Thus, we may assume that the elements $x_1,\hdots,x_t,\sigma$ have $p$-power order.

Suppose that $\sigma$ has orbit lengths $t_1,\hdots,t_m$, and write $\sigma=\sigma_1\hdots\sigma_m$ as a product of disjoint cycles, where $\sigma_i$ is a cycle of length $t_i$. Without loss of generality, assume that $\sigma_1=(1,\hdots,t_1)$, $\sigma_2=(t_1+1,\hdots,t_1+t_2)$, and so on. Set $g_1:=(x_1,\hdots,x_{t_1})\sigma_1$, $g_2:=(x_{t_1+1},\hdots,x_{t_1+t_2})\sigma_2$, etc. Then $g_i$ is a $p$-element of $\Aut(S)\wr\Sym_{t_i}$ for each $1\le i\le m$, and the coset $Ng$ is the cartesian product of the sets $N_1g_1\times\hdots\times N_mg_m$, where $N_i=S^{t_i}$. It clearly follows that $f_p(N)\geq f_p(N_1)\hdots f_p(N_m)$. Thus, we may assume that $m=1$ and that $\sigma=(1,\hdots,t)$ is a $t$-cycle. In particular, $t$ is a power of $p$.

Now, fix $y:=(y_1,\hdots,y_t)\sigma\in Ng$, $z:=(z_1,\hdots,z_t)\sigma\in Ng$, and $h:=(h_1,\hdots,h_t)\in N$, where $h_i,y_i,z_i\in \Aut(S)$. Then a routine exercise shows that $y^h=z$ if and only if $h_i^{-1}y_ih_{i+1}=z_i$ for each $1\le i\le t$, and $h_t^{-1}y_th_{1}=z_t$. In particular,
\begin{align}\label{lab:rel}
\text{if $y^h=z$, then $h_1^{-1}(y_1\hdots y_{t})h_1=z_1\hdots z_{t}$.}   
\end{align}
More generally, $h_i^{-1}(y_iy_{i+1}\hdots y_ty_1\hdots y_{i-1})h_1=z_iz_{i+1}\hdots z_tz_1\hdots z_{i-1}$ for all $i$. 
We also have $y^t=(y_1y_2\hdots y_t,y_2y_3\hdots y_ty_1,\hdots,y_ty_1\hdots y_{t-1})$. In particular,
\begin{align}\label{lab:rel2}
\text{$y$ has order $t|y_1\hdots y_t|$.}
\end{align}

Next, let $\mathcal{C}(S)$ be a set of representatives for the $S$-conjugacy classes of $p$-elements in the coset $Sx_1\hdots x_t\subseteq \Aut(S)$. Fix $\alpha\in \mathcal{C}(S)$ with $|\alpha^S|\geq M$, and write $\alpha=ax_1\hdots x_t$, with $a\in S$. Consider the element \begin{align*}x(\alpha):=(a,1,\hdots,1)g=(ax_1,x_2,\hdots,x_t)\sigma\in Ng.
\end{align*}
Since $\alpha=ax_1\hdots x_t$ is a $p$-element of $\Aut(S)$, (\ref{lab:rel2}) implies that $x(\alpha)$ is a $p$-element of $Ng$ (recall that $t$ is a power of $p$). Furthermore, (\ref{lab:rel}) implies that if $\alpha\neq \beta\in \mathcal{C}(S)$, then $x(\alpha)$ and $x(\beta)$ are not $N$-conjugate. Moreover, (\ref{lab:rel}) and the comments afterward imply that if $(h_1,\hdots,h_t)\in C_N(x(\alpha))$, then $h_1\in C_S(ax_1\hdots x_t)$ and $h_i\in $ $C_S(x_ix_{i+1}\hdots x_tax_1\hdots x_{i-1})$ for $i>1$. Since $ax_1\hdots x_t$ and $x_ix_{i+1}\hdots x_tax_1\hdots x_{i-1}$ are $\Aut(S)$-conjugate for $i>1$, we deduce that $|C_N(x(\alpha))|\le |C_S(ax_1\hdots x_t)|^t$. Hence, $|x(\alpha)^N|\geq |\alpha^S|^t\geq M^t$.

Finally, let $\mathcal{C}(N)$ be a set of representatives for the $N$-conjugacy classes of $p$-elements in the coset $Ng$, so that $|\Ord(Ng,p)|=\sum_{x\in\mathcal{C}(N)}|x^N|$. By our work in the previous paragraph, we have 
\begin{align*}
    |\Ord(Ng,p)|=\sum_{x\in\mathcal{C}(N)}|x^N|\geq \sum_{\alpha\in C(S)}|x(\alpha)^N|\geq \sum_{\alpha\in C(S)}|\alpha^S|^t\geq M^t.
\end{align*}
This completes the proof.
\end{proof}

We can see from Lemmas \ref{lem:CrucialRed} and \ref{lem:fpnonab} that lower bounds on the sizes of $S$-conjugacy classes of $p$-elements in $\Aut(S)$-cosets of nonabelian finite simple groups $S$ can be used to determine lower bounds on the number of $p$-elements in an arbitrary finite group. This is where our motivation for Theorems \ref{thm:MainTheoremp} and \ref{thm:MainTheoremr} come from. We will prove these theorems in Sections \ref{sec:prproofs} and \ref{sec:asymptotic} respectively.

We conclude this section with a useful corollary of Lemmas \ref{lem:CrucialRed} and \ref{lem:fpnonab}. To state it, we require the following notation: for a finite simple group $S$ and a prime $p$, define $M_p(S):=|S|$ if $S$ is a cyclic $p'$-group; $M_p(S):=1$ if $S$ is cyclic of order $p$; and if $S$ is a nonabelian finite simple group, then define $M_p(S)$ to be the minimum $M$ such that whenever $\alpha$ is a $p$-element of $\Aut(S)$, then the coset $S\alpha\subseteq \Aut(S)$ contains a $p$-element $x$ with $|C_S(x)|\le M$. Finally, for a finite group $G$, define $M_p(G)=\prod_{S\text{ a composition factor of }G}M_p(S)$.

We also need the following standard terminology: If $G_1,\hdots,G_t$ are finite groups, we say that $G$ is a \emph{subdirect subgroup} of the direct product $X:=G_1\times\hdots\times G_t$ if $G\le X$ and the projection map $\pi_i:G\rightarrow G_i$ is surjective for all $1\le i\le t$.
Our corollary can now be stated as follows.   
\begin{cor}\label{cor:MpCor}
Fix a prime $p$, and a finite group $G$. Then:
\begin{enumerate}[\upshape(i)]
    \item If $N\unlhd G$, then $M_p(G)=M_p(N)M_p(G/N)$;
    \item $M_p(G)\geq |G|/\Ord(G,p)$; and
    \item Suppose that $G$ is a subdirect subgroup of a direct product $G_1\times\hdots\times G_t$ of finite groups $G_i$. Then $M_p(G)\le \prod_{i=1}^tM_p(G_i)$.
\end{enumerate}
\end{cor}
\begin{proof}
Part (i) follows immediately from the definition of the function $M_p$, while (ii) follows from Lemmas \ref{lem:CrucialRed} and \ref{lem:fpnonab}. For (iii), note that if $G$ is a subdirect subgroup of $G_1\times\hdots\times G_t$, then the intersection of $G$ with the coordinate subgroup $\hat{G}_1:=\{(g_1,1,\hdots,1)\text{ : }g_1\in G_1\}$ is a normal subgroup of $G_1$. Thus, $M_p(G\cap \hat{G}_1)\le M_p(G_1)$. Since $M_p(G)=M_p(G\cap\hat{G}_1)M_p(G/G\cap\hat{G}_1)$, the result now follows from an easy inductive argument. 
\end{proof}

\section{Upper bounds on $M_p(S)$ when $S$ is an alternating group}\label{sec:AltThmSec}
To prove Theorems \ref{thm:JordanAnalogue}, \ref{thm:JordanIrr}, \ref{thm:MainTheoremPrimLin} and \ref{thm:MainTheoremPerm}, we need to find lower bounds on $\Ord(G,p)/|G|$ for particular linear and permutation groups $G$. By Corollary \ref{cor:MpCor}, it will suffice to find upper bounds on $M_p(G)$ in these cases. Thus, with the definition of the function $M_p$ in mind, we will need to determine upper bounds on $M_p(S)$ for the finite nonabelian simple groups $S$. These upper bounds will come from Theorems \ref{thm:MainTheoremp} and \ref{thm:MainTheoremr} (both of which will be proved in Section \ref{sec:prproofs}) when $S$ is not an alternating group. The purpose of this section is to determine an upper bound on $M_p(S)$, in terms of $p$ and $n$, when $S$ is an alternating group of degree $n$. Our main result reads as follows.
\begin{thm}\label{thm:AltThm}
Fix a prime $p$, and let $S\cong\Alt_n$ be an alternating group of degree $n\geq p$. Let $\alpha$ be a $p$-element of $\Aut(S)$. Then the coset $S\alpha\subseteq \Aut(S)$ contains an element $x$ with $|C_S(x)|\le p^{n/2}$. That is, $M_p(S)\le p^{n/2}$. 
\end{thm}

To prove Theorem \ref{thm:AltThm}, we will need two straightforward combinatorial lemmas.
\begin{lemma}\label{lem:basep}
Fix a prime $p$, and a positive integer $n\geq p$. Write the base $p$ expansion of $n$ in the form $n=\sum_{i=1}^{k}a_ip^{b_i}$, where $0<a_i\le p-1$ for each $i$, and $0\le b_1<\hdots<b_k\le \log_p{n}$. Then each of the following holds.
\begin{enumerate}[\upshape(i)]
    \item Suppose that either $p\geq 5$; or that $p=3$ and $b_i>1$ for some $i$; or that $p=2$ and $b_i>3$ for some $i$. Then $\sum_{i=1}^k(a_i+b_i)\le n/2$;
    \item If either $p=3$ and $b_i\le 1$ for all $i$; or $p=2$ and $b_i\le 3$ for all $i$, then $\prod_{i=1}^ka_i!p^{b_i}\le p^{n/2}$.
\end{enumerate}
It follows in particular that $\prod_{i=1}^ka_i!p^{b_i}\le p^{n/2}$.
\end{lemma}
\begin{proof}
We first prove (i). So assume that either $p\geq 5$; or that $p=3$ and $b_i>1$ for some $i$; or that $p=2$ and $b_i>3$ for some $i$. Set $g_p(b_i):=p^{b_i}-2-2b_i$. Then $g_p(0)=-1$; $g_p(1)=p-4$; $g_p(2)=p^2-6$; $g_p(3)=p^3-8$; and $g_p(4)=p^4-10$. Moreover, $g_p(b_i)\geq g_p(b_i-1)>0$ for $b_i\geq 5$. Note that $b_i\geq 1$ for some $i$, since $n\geq p$. It follows that $\sum_{i=1}^kg_p(b_i)\geq 0$ if either $p\geq 5$; or $p=3$ and $b_i>1$ for some $i$; or $p=2$ and $b_i>3$ for some $i$. Since we are assuming that these conditions hold, we deduce that $2\sum_{i=1}^kb_i\le \sum_{i=1}^k(p^{b_i}-2)\le \sum_{i=1}^ka_i(p^{b_i}-2)$, and hence that $\sum_{i=1}^k(a_i+b_i)\le \frac{1}{2}\sum_{i=1}^ka_ip^{b_i}=n/2$. This proves (i).

For (ii), notice that if $p=3$ and $b_i\le 1$ for all $i$, then $n$ has the form $n=x+3y$, where $0\le x\le 2$, and $y\in \{1,2\}$. Hence, $n$ is at most $8$. Similarly, if $p=2$ and $b_i\le 3$ for all $i$, then $n\le 15$. We can then quickly check the result for each $n\geq p$ in each case. Finally, since $\prod_{i=1}^ka_i!p^{b_i}\le \prod_{i=1}^kp^{a_i+b_i}$, the proof of the lemma is complete. 
\end{proof}

\begin{lemma}\label{lem:paritybase}
For a positive integer $n\geq 2$, write the base $2$ expansion of $n$ in the form $n=\sum_{i=1}^{k(n)}2^{b_i}$, where $0\le b_1<\hdots<b_k\le \log_2{n}$. Let $m$ be minimal so that $k(n)$ and $k(n-m)$ have different parities. Then $m=1$ if $n$ is odd, while $m=2$ if $n$ is even.
\end{lemma}
\begin{proof}
It is clear that if $n$ is odd (so that $b_1=0$), then $k(n-1)=k(n)-1$, whence $k(n)$ and $k(n-1)$ have different parities. So assume that $n$ is even, so that $b_1=1$. Then $k(n-1)=k(n)$, so $k(n-1)$ and $k(n)$ have the same parities. But the argument above implies that $k(n-1)$ and $k(n-2)$ have different parities, since $n-1$ is odd. Thus, $k(n)$ and $k(n-2)$ have different parities, and the proof is complete. 
\end{proof}

We can now prove Theorem \ref{thm:AltThm}.
\begin{proof}[Proof of Theorem \ref{thm:AltThm}]
Fix a prime $p$, and a $p$-element $\alpha$ of $\Aut(S)$, where $S=\Alt_n$ and $n\geq \max\{5,p\}$). To prove Theorem \ref{thm:AltThm}, it will suffice to prove that there is a $p$-element $x$ in the coset $S\alpha$ with the property that $|C_S(x)|\le p^{n/2}$.
If $p=2$ and $n=6$, then this can be checked using the Web Atlas \cite{WebAtlas}, so we will assume throughout that we are not in this case.

Write the base $p$ expansion of $n$ as $n=\sum_{i=1}^{k}a_ip^{b_i}$, where $0<a_i\le p-1$ for each $i$, and $0\le b_1<\hdots<b_k\le \log_p{n}$. Let $x:=x_p(n)$ be any corresponding element of $\Sym_n$. That is, $x$ is a permutation with $a_i$ cycles of length $p^{b_i}$, for each $1\le i\le k$. 
Then $|C_{\Sym_n}(x)|=\prod_{i=1}^ka_i!p^{b_i}$. It follows from Lemma \ref{lem:basep} that $|C_S(x)|\le |C_{\Sym_n}(x)|\le p^{n/2}$. 


Since $\Aut(S)/S$ is a $2$-group, this proves the proposition in all cases apart from when $p=2$ and $x:=x_2(n)\not\in S\alpha$. So assume that we are in this case, and write $k=k(n)$, so that $n=\sum_{i=1}^{k(n)}2^{b_i}$ is the base $2$ expansion of $n$. Since we are assuming that $n\neq 6$, the element $\alpha$ is contained in $\Sym_n$. Notice also that $x=x_2(n)$ is in $\Alt_n$ if and only if $k(n)$ is even. Choose $m\le n$ minimal with the property that $k(n-m)$ and $k(n)$ have different parities. Then by Lemma \ref{lem:paritybase}, either $n$ is odd and $m=1$, or $n$ is even and $m=2$. In the former case, it is clear that $x(n-1)$ has the same $\Sym_n$-centralizer as $x(n)$. In the latter case, we have $|C_{\Sym_n}(x(n-2))|=2|C_{\Sym_n}(x(n))|$, and $C_{\Sym_n}(x)$ contains a transposition. Setting $y:=x(n-1)$ if $n$ is odd and $y:=x(n-2)$ if $n$ is even, it follows that $y\in S\alpha$ and $|C_S(y)|\le |C_{\Sym_n}(x(n))|\le p^{n/2}$, with the latter inequality following from Lemma \ref{lem:basep} as in the previous paragraph. This completes the proof.
\end{proof}

\section{The proof of Theorem \ref{thm:MainTheoremp}}\label{sec:prproofs}
Having determined upper bounds for $M_p(S)$ when $S$ is an alternating group, our attention now turns to the case where $S$ is a finite group of Lie type in defining characteristic. Our upper bounds for $M_p(S)$ in this case comes from Theorem \ref{thm:MainTheoremp}, and the purpose of the current section is to prove this theorem. We can do so without further discussion. 
\begin{proof}[Proof of Theorem \ref{thm:MainTheoremp}]
Let $S$ be a finite simple group in $\mathrm{Lie}(p)$, and write $q$ and $\ell$ for the level and untwisted Lie rank of $S$, respectively. Fix a $p$-element $\alpha$ of $\Aut(S)$. We need to prove that there exists a $p$-element in the coset $S\alpha\subseteq \Aut(S)$ with $S$-centralizer order less than or equal to $q^{\ell}$.

We begin by fixing some Lie theoretic notation. Write $S=O^{p'}(\ol{X}_{\ol{\sigma}})$, where $\ol{X}$ is a simple algebraic group, $\ol{\sigma}:\ol{X}\rightarrow \ol{X}$ is a Steinberg endomorphism of $\ol{X}$, and $\ol{X}_{\ol{\sigma}}=\{x\in \ol{X}\text{ : }\ol{\sigma}(x)=x\}$. In particular, by \cite[Theorem 2.2.3]{GLS}, we may write $\ol{\sigma}=\ol{\sigma_p}^f\ol{\gamma}^i$, where $\ol{\sigma_p}$ is a Frobenius automorphism of $\ol{X}$, $q=p^f$, and either 
\begin{enumerate}[(a)]
    \item $\ol{X}\in\{A_{\ell},D_{\ell},E_6\}$ and $\ol{\gamma}$ is a graph automorphism of $\ol{X}$ as in \cite[Theorem 1.15.2(a)]{GLS};
    \item $\ol{X}\in\{B_2\text{ }(p=2),G_2\text{ }(p=3)\text{, }F_4\text{ }(p=2)\}\}$, and $\ol{\gamma}$ is as in \cite[Theorem 1.15.4(b)]{GLS}; or
    \item $\ol{X}$ is not one of the groups listed in (a) or (b) above, and $\ol{\gamma}=1$.
\end{enumerate}
Moreover, we may write $\alpha=x\sigma_p^e\gamma^j$, where $x\in \Inndiag(X):=\ol{X}_{\ol{\sigma}}$; $\sigma_p:=\ol{\sigma}_p\downarrow_S$; $0\le e<f$; and $\gamma:=\ol{\gamma}\downarrow_S$ (see \cite[Section 2.5]{GLS}). In fact, since $\alpha$ is a $p$-element of $\Aut(S)$, and since we are working in the coset $S\alpha$, we may assume that $\alpha=\sigma_p^e\gamma^j$.   

Now, as we will see later in the proof, the result will follow almost immediately, via the theory of Shintani descent, from our analysis in the cases (1) $\alpha=1$; and (2) $S\in\{A_{\ell}(q),D_{\ell}(q),E_6(q)\}$ and $\alpha$ is a graph automorphism of $S$. With this in mind, assume first that $\alpha=1$. We will show that there exists a unipotent element $x\in S$ such that $|C_S(x)|\le |C_{\ol{X}_{\ol{\sigma}}}(x)|\le 2q^{\ell}$. We will prove that this is the case when $x$ is chosen to be a regular unipotent element of $S$. Such an element $x$ exists in $S$ by \cite[Proposition 5.1.7(a)]{Carter}. Moreover, at least one such $x$ satisfies $|C_S(x)|\le|C_{\ol{X}_{\ol{\sigma}}}(x)|\le q^{\ell}|\mathcal{C}_{reg,p}|$ by \cite[Proposition 5.1.9]{Carter}, where $|\mathcal{C}_{reg,p}|$ is the number of conjugacy classes of regular unipotent elements in $\ol{X}_{\ol{\sigma}}$. If the prime $p$ is \emph{good} for $\ol{X}$, we have $|\mathcal{C}_{reg,p}|=1$ by \cite[Proposition 5.1.7(b)]{Carter} (see \cite[page 28]{Carter} for a list of the good and bad primes for each $\ol{X}$).   
Thus, we may assume that $p$ is bad for $\ol{X}$. We distinguish the classical and exceptional cases:
\begin{enumerate}
    \item $S$ is classical. The only bad prime here is $p=2$, and it only occurs in the cases $S\in\{\PSp_{2\ell}(q),\POmega^{\epsilon}_{2\ell}(q)\text{ : }\epsilon\in\{\pm\}\}$. So assume we are in one of these cases, and let $\lambda$ be an element of $\mathbb{F}_q$ such that the polynomial $x^2+x+\lambda$ is irreducible over $\mathbb{F}_q$. Write $\hat{S}$ for the quasisimple matrix group associated to $S$, so that $\hat{S}/Z(\hat{S})\cong S$. Then in the notation of \cite[Theorem 3.1 and its proof]{GLOB}, and by using \cite[Section 6.1]{LS}, we see that the following holds:
    \begin{enumerate}[(i)]
        \item $V_{\lambda}(2\ell)$ reduces modulo $Z(\hat{S})$ to a regular unipotent element of $S$ if $\hat{S}=\Sp_{2\ell}(q)$ (acting as a single Jordan block on the natural module for $\hat{S}$);
        \item $V(2\ell-2)+V(2)$ reduces modulo $Z(\hat{S})$ to a regular unipotent element of $S$ if $\hat{S}=\Omega^+_{2\ell}(q)$; and
        \item $V_{\lambda}(2\ell-2)+V(2)$ reduces modulo $Z(\hat{S})$ to a regular unipotent element of $S$ if $\hat{S}=\Omega^-_{2\ell}(q)$.
    \end{enumerate}
   It then follows from \cite[Proposition 3.2]{GLOB} $|\mathcal{C}_{reg,p}|\le 2$. Thus, $|C_S(x)|\le |C_{\ol{X}_{\ol{\sigma}}}(x)|\le 2q^{\ell}$, as claimed.
   
   \item $S$ is exceptional. The finite exceptional groups arising from bad primes for algebraic groups are precisely those given in Table \ref{tab:badprimes}. Upper bounds for the order of the centralizer of a regular unipotent element of $S$ are known in these cases. We give these upper bounds, together with the associated references, in Table \ref{tab:badprimes}. We also make two further remarks. First, suppose that $S\not\in\{E^{\pm}_6(q),E_7(q),E_8(q)\}$. To identify the regular unipotent elements in $S$ in each of the given references, we use \cite[Proposition 5.1.3]{Carter} (which gives a regular semisimple element as a word in the Chevalley generators for $\ol{X}$) and  \cite[Proposition 5.1.9]{Carter} (which tells us how many regular semisimple elements there are in $\ol{X}_{\ol{\sigma}}$). Note also that $S=\ol{X}_{\ol{\sigma}}$ in these cases. Our second remark concerns the case $S\in\{E^{\pm}_6(q),E_7(q),E_8(q)\}$. In these cases Mizuno \cite{MizunoE6,Mizuno} does not explicitly give the order of the centralizer of a regular semisimple element $x$: he simply gives $|C_{\ol{X}}(x)/|C_{\ol{X}}(x)^{\circ}|$ (written as $Z(x)$ in Mizuno's tables). By \cite[Proposition 5.1.2]{Carter}, the regular unipotent elements of $\ol{X}$ form a single conjugacy class in $\ol{X}$. Then, by \cite[Theorem 2.1.5]{GLS}, the number $|\mathcal{C}_{reg,p}|$ of regular unipotent classes in $\ol{X}_{\ol{\sigma}}$ is at most $|C_{\ol{X}}(x)/|C_{\ol{X}}(x)^{\circ}|$. Thus, as mentioned above, we have $|C_S(x)|\le |C_{\ol{X}_{\ol{\sigma}}}(x)|\le |C_{\ol{X}}(x)/|C_{\ol{X}}(x)^{\circ}|q^{\ell}$. For example, using \cite[Theorem 6.2]{MizunoE6}, we get $|C_{\ol{X}}(x)/|C_{\ol{X}}(x)^{\circ}|\le 6$. Thus, $C_S(x)\le 6q^6$. This is the upper bound given in Table \ref{tab:badprimes}.

   \begin{table}[]
       \centering
       \begin{tabular}{c|c|c|c}
        $S$    & $p$ &  $|C_{\ol{X}_{\ol{\sigma}}}(x)|\le $ & Reference\\
        \hline
        ${}^2B_2(q)$    & $2$  &  $2q$ & \cite[Propositions 1 and 18, and Theorem 6]{SuzukiSu1}\\
        ${}^2G_2(q)$    & $3$  & $3q$  & \cite[Chapter III, paragraph 11]{Ward}\\
        ${}^3D_4(q)$    & $2$  & $2q^4$ & \cite[0.5, page 677]{Spaltenstein}\\
        $G_2(q)$        &  $2$ & $2q^2$  & \cite[Table IV--1]{EnYa}\\
                        &  $3$ & $3q^2$  & \cite[Table VII--1]{En}\\
        ${}^2F_4(q)$    & $2$  & $4q^4$  & \cite[Table II]{Shinoda75}\\
        $F_4(q)$        &  $2$ & $4q^4$  & \cite[Theorem 2.1]{Shinoda74}\\
                        &  $3$ & $3q^4$  & \cite[Table 6]{Shoji}\\
        $E^{\pm}_6(q)$        &  $2,3$ & $6q^6$  & \cite[Theorem 6.2]{MizunoE6}\\
        $E_7(q)$        &  $2$ &  $4q^7$ & \cite[Table 9]{Mizuno}\\
                        &  $3$ &  $6q^7$ & \cite[Table 9]{Mizuno}\\
        $E_8(q)$        &  $2$ &   $4q^8$ & \cite[Table 10]{Mizuno}\\
                        &  $3$ &   $3q^8$ & \cite[Table 10]{Mizuno} \\ 
                        &  $5$ &   $5q^8$ & \cite[Table 10]{Mizuno}\\ 
                        \hline
       \end{tabular}
       \caption{Upper bounds for centralizer orders of regular unipotent elements of simple exceptional groups in bad characteristic.}
       \label{tab:badprimes}
   \end{table}
   
\end{enumerate}
This completes the proof in the case $S\alpha=S$.

Next, assume that $S\in\{A_{\ell}(q),D_{\ell}(q),E_6(q)\}$ and $\alpha$ is a graph automorphism of $S$. We then have either $p=2$ and $S\in\{A_{\ell}(q),D_{\ell}(q),E_6(q)\}$; or $p=3$ and $S=D_4(q)$. In these cases, the $S$-conjugacy classes of unipotent elements in $S\alpha$ and their centralizer orders are given in \cite{LLS}. The details of what we need are as follows. First, if $S=A_{\ell}(q)$ with $p=2$, then \cite[Theorem 1.3]{LLS} implies that $S\alpha$ contains a $2$-element $x$ with $|C_S(x)|\le 2q^{\ell/2}$ if $\ell$ is even, and  $|C_S(x)|\le 4q^{(\ell+1)/2}$ if $\ell$ is odd (setting $n=\ell+1$, and in the notation of \cite[(1)]{LLS}, we take $x$ corresponding to the decompositions $V\downarrow_x=V(2n)$ and $V\downarrow_x=V(2)+V(2n-2)$, respectively). Since $\ell$ is (necessarily) at least $2$, this gives us what we need unless $q=\ell=2$. However, in this case, we have $S\cong \Lg_3(2)$ and there is an element of order $8$ in $S\alpha$ with $|C_S(\alpha)|=|C_{\ol{X}_{\ol{\sigma}}}(x)|=4=q^{\ell}$.  

Suppose next that $S=D_{2\ell}(q)$ with $p=2$. Let $\hat{S}=\Omega^{+}_{2\ell}(q)$ be the associated quasisimple matrix group. We may then choose an element $\hat{\alpha}$ of $\mathrm{O}_{2\ell}^{+}(q)\setminus\Omega^{+}_{2\ell}(q)$ so that $\hat{S}\langle\hat{\alpha}\rangle$ reduces modulo the scalar subgroup of $\mathrm{O}_{2\ell}^{+}(q)$ to the group $S\langle\alpha\rangle$. Now, \cite[Lemma 3.3]{GLOB} implies that $\hat{S}\hat{\alpha}\subseteq \mathrm{O}_{2\ell}^{+}(q)\subseteq \mathrm{Sp}_{2\ell}(q)$ contains a unipotent element $x$ ($x=V(2k)$ in the notation of \cite[Lemma 3.3]{GLOB}) which acts as a single Jordan block on the natural module for $\mathrm{Sp}_{2\ell}(q)$. As in case (1) of the $S\alpha=S$ case above, we can then use \cite[Proposition 3.2]{GLOB} to conclude that $|C_S(x)|=|C_{\ol{X}_{\ol{\sigma}}}(x)|\le |C_{\Sp_{2\ell}(q)}(x)|\le 2q^{\ell}$. This gives us what we need. 

Finally, if $(S,p)=(D_4(q),3)$ or $(S,p)=(E_6(q),2)$, then $S\alpha$ contains a $p$-element $x$ with $|C_{\ol{X}_{\ol{\sigma}}}(x)|= q^2$ or $|C_{\ol{X}_{\ol{\sigma}}}(x)|= 2q^4$, respectively (see \cite[Tables 9 and 10]{LLS}). 

Thus, all that remains is to prove the theorem when $\alpha$ is non-trivial, and when we are not in the case where $S\in\{A_{\ell}(q),D_{\ell}(q),E_6(q)\}$ and $\alpha$ is a graph automorphism of $S$. Then $\alpha$ is the restriction to $S=O^{p'}(\ol{X}_{\ol{\sigma}})$ of a Steinberg endomorphism $\ol{\alpha}$ of $\ol{X}$. In these cases, the theory of Shintani descent (see for example, \cite{Harper}) gives us what we need. Indeed, by \cite[Theorem 2.1 and Lemma 2.16]{Harper}, the coset $S\alpha$ contains a unipotent element $x$ with the property that $|C_{\ol{X}_{\ol{\sigma}}}(x)|=|C_{\ol{X}_{\ol{\sigma_1}}}(x_1)|$, where
\begin{itemize}
    \item $\ol{\sigma_1}$ is a Steinberg endomorphism of $\ol{X}$ of the form $\ol{\sigma_1}=\bar{\sigma_p}^{f_1}\bar{\gamma}^{i_1}$, with $f_1\le f$ and $[\ol{\sigma},\ol{\sigma_1}]=1$; and
    \item $x_1$ is a $p$-element contained in $S_1\alpha_1$, where $S_1:=O^{p'}(\ol{X}_{\ol{\sigma_1}})$, and $\alpha_1$ is either trivial; or $S\in\{A_{\ell}(q),D_{\ell}(q),E_6(q)\}$ and $\alpha_1$ is a graph automorphism of $S_1$ (see (a) in the first paragraph of the proof). 
\end{itemize}
The result therefore follows from our work in the preceding paragraphs.
\end{proof}

\section{Results on primitive linear groups}\label{sec:primsec}
In this section, we make a short but necessary detour to discuss primitive linear groups. Recall that an irreducible linear group $G\le \mathrm{GL}(V)$ is said to be \emph{imprimitive} if the natural $G$-module $V$ has a direct sum decomposition $V=V_1\oplus\hdots\oplus V_t$ with $t>1$, such that for each $1\le i\le t$, and each $x\in G$, either $V_i^x=V_i$ or $V_i^x\cap V_i=0$. If no such decomposition exists, then the irreducible group $G\le \mathrm{GL}(V)$ is said to be \emph{primitive}.

We also need to define two linear group properties which are weaker than primitivity. First, recall that a linear group $N\le\GL(V)$ is said to be \emph{homogeneous} if $V\downarrow_N$ is completely reducible with pairwise isomorphic irreducible constituents. A linear group $G\le \GL(V)$ with the property that every normal [respectively characteristic] subgroup $N$ of $G$ is \emph{homogeneous} is called \emph{quasiprimitive} [resp. weakly quasiprimitive]. Since the homogenous components of a normal subgroup of an irreducible linear group $G\le \GL(V)$ form an imprimitivity decomposition for $V$, a primitive linear group is both quasiprimitive and weakly quasiprimitive. 

The following series of results give details on the structure of a weakly quasiprimitive linear group.
\begin{lemma}\label{lem:qlin}
Let $q$ be a power of a prime $p$, and let $G\le \GL(V)=\GL_n(q)$ be weakly quasiprimitive. Then there exists a divisor $d$ of $n$, and a subnormal series $G_a\unlhd G_{a-1}\unlhd\hdots\unlhd G_0=:G$, such that each of the following holds:
\begin{enumerate}[\upshape(i)]
    \item $G_i/G_{i+1}$ is abelian, for each $1\le i\le a-1$;
    \item $\left(\prod_{i=1}^{a-1}|G_i/G_{i+1}|\right)$ divides $d$.
    \item $G_a$ is isomorphic to a subgroup of $\GL_{n/d}(q^d)$ satisfying each of the following properties:
    \begin{enumerate}[\upshape(a)]
        \item The generalized Fitting subgroup $F^*(G_a)$ of $G_a$ is a homogeneous subgroup of $\GL_{n/d}(q^d)$.
        \item Let $\mathcal{R}$ be the set of primes $r$ such that $O_r(G_a)\not\le Z(G_a)$, and let $\mathcal{S}$ be a set of representatives for the conjugacy classes of subnormal quasisimple subgroups of $G_a$. Also, for $S\in\mathcal{S}$, set $T(S):=\langle S^g\text{ : }g\in G_a\rangle$. Then $F^*(G_a)$ is a central product of $Z(G_a)$, the $O_{r}(G_a)$ for $r\in \mathcal{R}$, and the $T(S)$ for $S\in\mathcal{S}$. Each $O_{r}(G)$ is extraspecial of exponent $r(2,r)$. Furthermore, each $T(S)$ is a central product of $t(S)\geq 1$ copies of $S$.
        \item $Z(G_a)\le Z(\GL_{n/d}(q^d))$.
        \item Let $V_0$ be the natural $\GL_{n/d}(q^d)$-module. If $U$ is an irreducible constituent of $V_0\downarrow_{F^*(G_a)}$, then $U$ is a tensor product of a $1$-dimensional module for $Z(G_a)$, irreducible modules $U(r)$ for each $O_{r}(G)$ with $r\in\mathcal{R}$, and irreducible modules $U(T)$ for each $T=T(S)\in\mathcal{S}$. Moreover, each $U(r)$ is an absolutely irreducible $\mathbb{F}_{q^d}[O_{r}(G_a)]$-module; while each $U(T)$ is a tensor product of $t:=t(S)$ copies of an absolutely irreducible $\mathbb{F}_{q^d}[S]$-module $U(S)$.
    \end{enumerate}
\end{enumerate}
\end{lemma}
\begin{proof}
Let $G$ be a counterexample to the lemma with $n$ minimal. Note first that since $G$ acts homogeneously on $V$, $G$ acts faithfully on each of its irreducible constituents. Thus, we may assume that $G$ is irreducible.

Now, suppose first that $G$ has a characteristic abelian subgroup which is not contained in $Z(\GL_n(q))$. Then by \cite[Lemma 2.13]{HRD}, there exists a divisor $d_1$ of $n$, and a characteristic subgroup $G_1$ of $G$, such that $G_1$ is isomorphic to a weakly quasiprimitive subgroup of $\GL_{n/d_1}(q^{d_1})$, and $G/G_1$ is abelian of order dividing $d_1$. If $d_1>1$, then the lemma holds for $G_1$, by the minimality of $G$ as a counterexample. But then the lemma clearly follows for $G$ -- a contradiction.

Thus, we may assume that all characteristic abelian subgroups of $G$ are contained in $Z(\GL_n(q))$. Let $\mathcal{R}$ be the set of primes $r$ such that $O_r(G)\not\le Z(G)$, and let $\mathcal{S}$ be a set of representatives for the conjugacy classes of subnormal quasisimple subgroups of $G$. Also, for $S\in\mathcal{S}$, set $T(S):=\langle S^g\text{ : }g\in G\rangle$. Then by \cite[Lemmas 2.14 and 2.16]{HRD}, $F^*(G)$ is a central product of $Z(G_a)$, the $O_{r}(G)$ for $r\in \mathcal{R}$, and the $T(S)$ for $S\in\mathcal{S}$. Each $O_{r}(G)$ is extraspecial of exponent $r(2,r)$. Furthermore, each $T(S)$ is a central product of $t(S)\geq 1$ copies of $S$.

Now, $F^*(G)$ is homogeneous, since $G$ is weakly quasiprimitive. Suppose that the irreducible constituents of $V\downarrow_{F^*(G)}$ are not absolutely irreducible. Then it is proved in \cite[Section2]{HROB} that there exists a divisor $d$ of $n$, and a normal subgroup $G_1$ of $G$ containing $F^*(G)$, such that:
\begin{itemize}
    \item $G_1$ is isomorphic to a subgroup of $\GL_{n/d}(q^{d})$;
    \item $V_0\downarrow_{F^*(G)}$ is homogeneous, with absolutely irreducible constituents; and
    \item $G/G_1$ is abelian of order dividing $d$.
\end{itemize}
(We remark that in \cite[Section 2]{HROB}, the authors assume that $G$ is absolutely irreducible. However, they only use this at the end of their argument to prove that $G/G_1$ is non-trivial. We do not require this additional assertion here, so we may use their arguments to conclude the points above.)
Note that since $F^*(G)\le G_1\unlhd G$, we have $F^*(G)=F^*(G_1)$.

Now, let $q_1=q^d$, and let $\mathbb{F}$ be a finite extension field of $\mathbb{F}_{q_1}$ such that $\mathbb{F}$ is a splitting field for each of the central factors $O_{r}(G_1)$, $S$ of $F^*(G_1)$, for $r\in\mathcal{R}$, $S\in\mathcal{S}$. Let $U$ be an irreducible $\mathbb{F}[F^*(G_1)]$-constituent of $V_0\downarrow_{F^*(G_1)}$. Then by \cite[Lemma 2.15]{HRD}, $U$ decomposes as a tensor product of a $1$-dimensional module $U_0$ for $Z(G_1)$, irreducible modules $U(r)$ for each $O_{r}(G)$, and irreducible modules $U(T)$ for each $T(S)$. The modules $U(r)$ are necessarily absolutely irreducible (over $\mathbb{F}$). Moreover, by \cite[Lemma 2.17]{HRD}, each $U(T)$ is a tensor product of $t(S)$ copies of a faithful absolutely irreducible module $U(S)$ for $S$. 

We claim that if $U'$ is one of the modules $U'=U(r)$ or $U'=U(S)$, then $U'$ is realizable over $\mathbb{F}_{q_1}$. Indeed, write $\chi$ for the Brauer character of $U$, and let $\chi_0$, $\chi_{r}$, and $\chi_{S}$ be the characters of the modules $U_0$, $U(r)$ and $U(S)$ respectively. Then $\chi=\chi_0(\prod_{r\in\mathcal{R}}\chi_{r})(\prod_{S\in\mathcal{S}}\chi_{S}^{t(S)})$. It follows that the extension field $\mathbb{F}_{q_1}(\chi_{r})$ of $\mathbb{F}_{q_1}$ generated by the values of $\chi_{r}$ is contained in $\mathbb{F}_{q_1}(\chi)$. Similarly, $\mathbb{F}_q(\chi_{S})$ is contained in $\mathbb{F}_{q_1}(\chi)$. Since $U$ is realizable over $\mathbb{F}_{q_1}$, \cite[Theorem 74.9]{CR} implies that $\mathbb{F}_{q_1}(\chi)=\mathbb{F}_{q_1}$. It follows that the character values of the modules $U(r)$ and $U(S)$ are all contained in $\mathbb{F}_{q_1}$. Another application of \cite[Theorem 74.9]{CR} then proves the claim: the modules $U(r)$ and $U(S)$ are realizable over $\mathbb{F}_{q_1}$. This completes the proof.  
\end{proof}

Our next lemma gives more information on the central factors of $F^*(G_a)$, where $G_a$ is as in Lemma \ref{lem:qlin}.
\begin{lemma}\cite[Lemmas 2.14--2.17]{HRD}\label{lem:MainPrimLin}
Let $q$ be a power of a prime $p$, and let $G\le \GL(V)=\GL_n(q)$ be weakly quasiprimitive. Let $d$, $G_a$, $\mathcal{R}$, $\mathcal{S}$, $U(r)$, and $U(T)$ be as in Lemma \ref{lem:qlin}. Also, set $q_1:=q^d$, and define $\mathcal{X}$ to be the set of subgroups $X$ of $G_a$ which either have the form $X=O_r(G_a)$ for $r\in\mathcal{R}$, or $X:=T(S)$ with $S\in\mathcal{S}$ Then the following assertions hold:
\begin{enumerate}[\upshape(i)]
        \item 
        Fix $r\in\mathcal{R}$, so that $O_r(G_a)$ is extraspecial, of order $r^{1+2m_r}=r^{1+2m}$ say. Then $G_a/C_{G_a}(O_r(G_a))O_r(G_a)$ is isomorphic to a completely reducible subgroup $L$ of $\Sp_{2m}(r)$, acting naturally on $O_r(G_a)/Z(O_r(G_a))\cong r^{2m}$. Moreover, $G_a/C_{G_a}(O_r(G_a))$ has shape $r^{2m}.L$, and $\dim_{\mathbb{F}_q}{U(r)}=r^m$.
        \item 
        Fix $S\in\mathcal{S}$, so that $T(S)$ is a central product of $t=t(S)$ copies of $S$. Then $G_a/C_{G_a}(T(S))$ is isomorphic to a subgroup of the wreath product $A(S)\wr \Sym_t$, where $A(S)$ is the stabilizer in $\Aut(S)$ of the module $U(S)$. Moreover, $G_a/C_{G_a}(T(S))$ contains $[S/Z(S)]^t$, and $\dim_{\mathbb{F}_q}{U(T)}=(\dim_{\mathbb{F}_q}{U(S)})^{t(S)}$.
    \end{enumerate}
In particular, writing $s(S):=\dim_{\mathbb{F}_{q_1}}{U(S)}$, we see that $\left(\prod_{r\in\mathcal{R}}r^{m_r}\right)\left(\prod_{S\in\mathcal{S}}s(S)^{t(S)}\right)$ divides $n$. 
\end{lemma}

The following two consequences of Lemmas \ref{lem:qlin} and \ref{lem:MainPrimLin} will be useful in our proofs of Theorems \ref{thm:JordanAnalogue} and \ref{thm:MainTheoremr}, respectively.
\begin{cor}\label{cor:PrimMpsCor}
Let $q$ be a power of a prime $p$, and let $G\le \GL(V)=\GL_n(q)$ be weakly quasiprimitive. Adopt the notation of Lemmas \ref{lem:qlin} and \ref{lem:MainPrimLin}. Then $M_p(G_a/Z(G_a))\le \prod_{X\in\mathcal{X}}M_p(G_a/C_{G_a}(X))$. \end{cor}
\begin{proof}
Adopting the notation in the statement of the corollary, we have $Z(G_a)=C_{G_a}(F^*(G_a))=\bigcap_{X\in\mathcal{X}}C_G(X)$ by standard theory of generalized Fitting subgroups. Writing $\mathcal{X}=\{X_1,\hdots,X_m\}$, it follows that $G_a/Z(G_a)$ embeds as a subdirect subgroup in the direct product $G_a/C_{G_a}(X_1)\times\hdots\times G_a/C_{G_a}(X_m)$. The result now follows immediately from Corollary \ref{cor:MpCor}.
\end{proof}

\begin{cor}\label{cor:PrimSimpleCor}
Let $q$ be a power of a prime $p$ and let $G\le \GL(V)=\GL_n(\mathbb{F})$ be weakly quasiprimitive. Suppose that $N$ is a normal subgroup of $G$ contained in $\mathrm{Frat}(G)$, with $G/N$ a nonabelian simple group. Then one of the following holds:
\begin{enumerate}[\upshape(i)]
    \item There exists a prime power divisor $r^m$ of $n$ such that $G/N$ is a section of $\Sp_{2m}(r)$; or
    \item There exists a divisor $s^t$ of $n$, with $s\geq 2$, such that either:
    \begin{enumerate}[\upshape(a)]
        \item $G/N$ has a projective irreducible representation of dimension $s$ over $\ol{\mathbb{F}}$; or
        \item $G/N$ is a section of $\Sym_t$.
    \end{enumerate}
\end{enumerate}
\end{cor}
\begin{proof}
Consider the subnormal series $G_a\unlhd\hdots\unlhd G_0=:G$ from Lemma \ref{lem:qlin}. Since the groups $G_i/G_{i+1}$ are abelian and $G/N$ is nonabelian simple, we must have $G=G_{2}N=\hdots=G_aN$. But then we must have $G_a=G$, since $N\le\mathrm{Frat}(G)$.   

So we have proved that $G=G_a$. Adopt the notation of Lemmas \ref{lem:qlin} and \ref{lem:MainPrimLin}, and let $X$ be a subgroup of $G$ which either has the form $X=O_r(G)$ for some $r\in\mathcal{R}$, or $X=T(S)$ for some $S\in\mathcal{S}$. Then since $N\le\mathrm{Frat}(G)$ and $C_G(X)\neq G$, we have $NC_G(X)\neq G$, whence $G/NC_G(X)\cong G/N$. Proving the claim is now just a matter of inspecting the nonabelian simple quotients of the groups $X$: If $X=O_r(G)\cong r^{1+2m}$ for some prime $r$, then $G/C_G(X)\cong r^{2m}.L$, where $L$ is a completely reducible subgroup of $\Sp_{2m}(r)$, and $r^m$ divides $n$. It follows that $G/N$ is a section of $\Sp_{2m}(r)$. If $X=T(S)$ for some quasisimple group $S$, then $G/C_G(X)$ is a subgroup of $A\wr \Sym_t$ containing $[S/Z(S)]^t$, where $A$ is almost simple with socle $S/Z(S)$. Since $A/S$ is soluble by Schreier's lemma, it follows that $G/N$ is either a section of $\Sym_t$, or isomorphic to $S$. The result now follows from Lemma \ref{lem:MainPrimLin}.   
\end{proof}

\section{The proofs of Theorems \ref{thm:MainTheoremPerm} and \ref{thm:MainTheoremr}}\label{sec:asymptotic}
Having recorded the information we need on primitive linear groups in Section \ref{sec:primsec}, we can now get back to the task of proving our main theorems. We remind the reader that our main tool in proving Theorems \ref{thm:JordanAnalogue} and \ref{thm:MainTheoremPrimLin} is Corollary \ref{cor:MpCor}, and more generally, the function $M_p$ and its behaviour on the nonabelian finite simple groups. We have already determined upper bounds on $M_p(S)$ when $S$ is an alternating group or a group of Lie type in characteristic $p$. When $S$ is sporadic or in $\mathrm{Lie}(r)$ for $r\neq p$, our bound on $M_p(S)$ will come from Theorem \ref{thm:MainTheoremr}. For this reason, the second half of the current section will be dedicated to the proof of this theorem.

In the first half of the section, and partly because it is needed as part of the proof of Theorem \ref{thm:MainTheoremr}, we will prove Theorem \ref{thm:MainTheoremPerm}. 
Our first preparatory result in this direction is similar to Lemma \ref{lem:basep}, and will help us to prove Theorem \ref{thm:MainTheoremPerm} in the case when the permutation group in question is the full alternating or symmetric group.
\begin{lemma}\label{lem:basepbound}
Fix a prime $p>3$ and a positive integer $n\geq 2$. Write the base $p$ expansion of $n$ as $n=\sum_{i=1}^{k}a_ip^{b_i}$, where $0<a_i\le p-1$ for each $i$, and $0\le b_1<\hdots<b_k\le \log_p{n}$. Then $\prod_{i=1}^ka_i!p^{b_i}\le (p-1)!^{\frac{n-1}{p-2}}$. Furthermore, if $a_i=1$ for all $i$, then $\prod_{i=1}^ka_i!p^{b_i}\le \frac{1}{2}(p-1)!^{\frac{n-1}{p-2}}$.
\end{lemma}
\begin{proof}
We will begin by proving the first statement.
Note that $\prod_{i=1}^ka_i!p^{b_i}$ is the order of a centralizer in $\Sym_n$, so $\prod_{i=1}^ka_i!p^{b_i}\le n!$. Note also that since $(m+1)!\le (m+1)^m$, the function $m\rightarrow m!^{\frac{1}{m-1}}$ is non-decreasing on the set of natural numbers. Hence, if $p>n$, then we have $\prod_{i=1}^ka_i!p^{b_i}\le n!= n!^{\frac{n-1}{n-1}}\le (p-1)!^{\frac{n-1}{p-2}}$.

Thus, we may assume that $p\le n$. Now, since $p>3$, we have $p\le (p-1)!$. Since $a_i\le p-1$, it follows that
\begin{align}\label{lab:not91}
\prod_{i=1}^ka_i!p^{b_i}\le (p-1)!^{\sum_{i=1}^k(1+b_i)}.
\end{align}
Notice also that
\begin{align}\label{lab:not92}
\frac{n-1}{p-2}\geq \frac{n-1}{p-1}\geq\sum_{i=1}^{k}\frac{p^{b_i}-1}{p-1}.
\end{align}
Now, $(p^{b_i}-1)/(p-1)=\sum_{\ell=0}^{b_i-1}p^{\ell}\geq (b_i-1)p+1$. Since $p\geq 5$, we have $(b_i-1)p+1\geq b_i+4$ when $b_i\geq 2$. It follows that if $b_j\geq 2$ for some $j$, then $\sum_{i=1}^{k}(p^{b_i}-1)/(p-1)=\sum_{b_i\in\{0,1,b_j\}}(p^{b_i}-1)/(p-1)+ \sum_{b_i\not\in\{0,1,b_j\}}(p^{b_i}-1)/(p-1)\geq b_j+4+\sum_{b_i\not\in\{0,1,b_j\}}(1+b_i)\geq \sum_{i=1}^k(1+b_i)$. Thus, if $b_j\geq 2$ for some $j$, then the result follows from (\ref{lab:not91}) and (\ref{lab:not92}).  

So we may assume that $n\geq p$ either has the form $n=a_1p$, or $n=a_1+a_2p$. If the former case holds, then $a_1\le p-1$, and so $a_1!p\le p!$. Moreover, either $a_1=1$ (and the bound  $a_1!p=p\le (p-1)!^{\frac{p-1}{p-2}}=(p-1)!^{\frac{n-1}{p-2}}$ clearly holds), or $a_1\geq 2$, and so $(p-1)!^{\frac{n-1}{p-2}}\geq (p-1)!^2\geq p!$, which again gives us what we need.  

Suppose next that $n=a_1+a_2p$. We need to prove that $a_1!a_2!p\le (p-1)!^{\frac{n-1}{p-2}}$. If $a_2\geq 3$, then $(p-1)!^{\frac{n-1}{p-2}}\geq (p-1)!^3\geq (p-1)!p!\geq a_1!a_2!p$. If $a_2=2$, then $(p-1)!^{\frac{n-1}{p-2}}\geq (p-1)!^2\geq 2p!\geq a_1!a_2!p$. If $a_2=1$ and $a_1\le p-3$, then $a_1!a_2!p\le (p-3)!p\le (p-1)!\le (p-1)!^{\frac{n-1}{p-2}}$, since $p\geq 5$. If $a_2=1$ and $a_1>p-3$, then $(p-1)!^{\frac{n-1}{p-2}}>(p-1)!^{\frac{p-3+p-1}{p-2}}=(p-1)!^2\geq p!\geq a_1!a_2!p$. The result follows in each case.  

Assume finally that $a_i=1$ for all $i$. Since $p\geq 5$, we have $p\le \frac{1}{2}(p-1)!$. Then $\prod_{i=1}^ka_i!p^{b_i}=p^{\sum_{i=1}^kb_i}\le \frac{1}{2}(p-1)!^{\sum_{i=1}^kb_i}$. Since $(n-1)/(p-2)=(\sum_{i=1}^kp^{b_i}-1)/(p-2)\geq \sum_{i=1}^k(p^{b_i}-1)/(p-1)\geq \sum_{i=1}^k b_i$, the result follows.
\end{proof}

We are now ready to prove Theorem \ref{thm:MainTheoremPerm}.
\begin{proof}[Proof of Theorem \ref{thm:MainTheoremPerm}]
Let $G$ be a counterexample to the theorem with $n$ minimal. Since $M_p(G)\le |G|\le n!$, we may assume throughout that $n\geq p$.
Suppose first that $G\in\{\Alt_n,\Sym_n\}$. If $p=2=n$, then $M_p(G)=1<3^{(n-1)/2}$. If $p=2$ and $n>2$, then $M_p(G)=M_p(\Alt_n)\le 2^{n/2}\le 3^{(n-1)/2}$, by Theorem \ref{thm:AltThm}. If $p=3$, then we can check that $M_p(G)\le 20^{(n-1)/4}$ for $2\le n\le 7$. If $n\geq 8$, then $3^{n/2}2\le 20^{(n-1)/4}$, so $M_p(G)\le 2M_p(\Alt_n)\le 3^{n/2}2\le 20^{(n-1)/4}$, again using Theorem \ref{thm:AltThm}.

Assume now that $p>3$. If $n=6$, then we have $p=5$ and $M_p(G)\le 4M_p(\Alt_6)=20\le (5-1)!^{(6-1)/(5-2)}$. So assume that $n\neq 6$. Since $n\geq p\geq 5$, it follows that $M_p(G)\le 2M_p(\Alt_n)$. Now, write the base $p$ expansion of $n$ as $n=\sum_{i=1}^{k}a_ip^{b_i}$, where $0<a_i\le p-1$ for each $i$, and $0\le b_1<\hdots<b_k\le \log_p{n}$. Let $x:=x_p(n)$ be any corresponding element of $\Sym_n$. That is, $x$ is a permutation with $a_i$ cycles of length $p^{b_i}$, for each $1\le i\le k$. Set $d:=1$ if $a_i=1$ for all $i$, and $d:=2$ otherwise. Then $|C_{\Alt_n}(x)|=\frac{1}{d}\prod_{i=1}^ka_i!p^{b_i}$. Since $p>3$, Lemma \ref{lem:basepbound} and the definition of the function $M_p$ implies that $M_p(G)\le 2M_p(\Alt_n)\le 2|C_{\Alt_n}(x)|\le (p-1)!^{(n-1)/(p-2)}$. 

Suppose next that $G$ is primitive, but that $G\not\in\{\Alt_n,\Sym_n\}$. Then by \cite[Theorem 1.1]{Maroti}, one of the following holds:
\begin{enumerate}[(i)]
    \item $G$ is a subgroup of $\Sym_m\wr\Sym_r$ containing $\Alt_m^r$ (with $m\geq 5$, $r\geq 2$), where the action of $\Sym_m$ is on $k$-element subsets of $\{1,\hdots,m\}$ (with $1\le k\leq m/2$), and the wreath product has the product action of degree $n=\binom{m}{k}^r$;
    \item $G\in\{M_{11},M_{12},M_{23},M_{24}\}$ and $G$ is $4$-transitive; or
    \item $|G|\le n^{1+\log_2{n}}$. 
\end{enumerate}
Suppose first that we are in case (i), and let $N$ be the intersection of $G$ with the base group of the wreath product. Then $N$ is a subdirect subgroup in a direct product of the form $N_1\times\hdots\times N_r$, where $N_i\in\{\Alt_m,\Sym_m\}$ for each $i$. It then follows from Corollary \ref{cor:MpCor} and our work in the previous paragraphs that $M_p(G)=M_p(N)M_p(G/N)\le (p-1)!^{(m-1)r/(p-2)}M_p(G/N)$. Since $G$ is a counterexample to the theorem of minimal degree, we also have $M_p(G/N)\le (p-1)!^{(r-1)/(p-1)}$. Thus, $M_p(G)\le (p-1)!^{(mr-1)/(p-2)}\le (p-1)!^{(n-1)/(p-2)}$, with the latter inequality following since $mr\le m^r\le  \binom{m}{k}^r=n$.

In case (ii), we can use direct computation to check that in each of the listed Matheiu groups, we have $M_p(G)\le (p-1)!^{(n-1)/(p-2)}$ whenever $5\le p\le n$. 

Assume now that we are in case (iii). As noted previously, the function $m\rightarrow m!^{1/(m-1)}$ is increasing on $\mathbb{N}$. Thus, $2^{n-1}\le (p-1)!^{(n-1)/(p-2)}$ for $p\geq 2$. By \cite[Corollary 1.4]{Maroti}, we have $M_p(G)\le |G|\le 2^{n-1}$, unless $G$ is one of the groups of degree at most $24$ listed in \cite[Corollary 1.4(i)--(iv)]{Maroti}. In these listed cases, we can check that $M_p(G)\le (p-1)!^{(n-1)/(p-2)}$ whenever $5\le p\le n$. The result follows.

Suppose next that $G$ is imprimitive. Then $G$ may be embedded as as a subgroup of a wreath product $R\wr \Sym_s$, where $R\le \Sym_r$ is primitive, $r,s\geq 2$, $rs=n$, and $G\cap R^s$ is isomorphic to a subdirect subgroup of a direct product $A^s$ of $s$ copies of a subnormal subgroup $A$ of $R$. Since $M_p(A)\le M_p(R)$ by definition of the function $M_p$, it follows from Corollary \ref{cor:MpCor} that $M_p(G)\le M_p(R)^sM_p(G/G\cap R^s)$. Since $G$ is a counterexample of minimal degree, we have $M_p(G/G\cap R^s)\le h_p(s)$ and $M_p(R)\le h_p(r)$. We deduce that $M_p(G)\le h_p(r)^sh_p(s)$. Recalling that $h_p(m):=3^{(m-1)/2}$ if $p=2$;  $h_p(m):=20^{(m-1)/4}$ if $p=3$; $h_p(m):=(p-1)!^{(m-1)/(p-2)}$ if $3<p\le n$, the result follows.

Finally, assume that $G$ is intransitive. Then $G$ can be embedded as a subgroup of $\Sym_r\times \Sym_s$, where $r\geq 2$ and $r+s=n$. Setting $N:=G\cap (\Sym_r\times 1)$, the minimality of $G$ as a counterexample, together with Corollary \ref{cor:MpCor}, implies that $M_p(S)\le h_p(r)h_p(s)$. Since $h_p(r)h_p(s)=h_p(r+s)$, the result follows. 
\end{proof}

It will be useful for us to note the following corollary of Theorem \ref{thm:MainTheoremPerm} (and its proof).
\begin{cor}\label{cor:MainTheoremPermCor}
Let $G\le\Sym_n$ be a permutation group of degree $n\geq 2$. Then $M_p(G)\le h_p(n)\le p^{n-1}$.
\end{cor}
\begin{proof}
In the proof of Theorem \ref{thm:MainTheoremPerm}, we in fact proved that $M_p(G)\le h_p(n)$. Since $(p-1)!\le (p-1)^{p-2}$, we can see from the definition of the function $h_p$ that $h_p(n)\le p^{n-1}$. 
\end{proof}

We now move on to our preparations for the proof of Theorem \ref{thm:MainTheoremr}. These preparations consist of the next two lemmas, both of which follow immediately from known results on simple groups coming from the CFSG. The first concerns minimal degrees of permutation representations, while the second concerns minimal dimensions of projective representations.  
\begin{lemma}\label{lem:PermSimple}
Let $S$ be a nonabelian finite simple group in $\mathrm{Lie}(p)$, and write $P(S)$ for the minimal degree of a non-trivial permutation representation for $S$. Then $|S|\le P(S)^{2\log_2{P(S)}}$.
\end{lemma}
\begin{proof}
The values of the function $P(S)$ are well-known -- see \cite[Table 4]{Guest}, for example. The bound in the lemma then follows immediately from inspection of the values of $|S|$ and $P(S)$ in each case. 
\end{proof}

\begin{lemma}\label{lem:RepSimple}
Fix a prime $p$, and let $S$ be a nonabelian finite simple group with $S\not\in\mathrm{Lie}(p)$. Write $n:=R_p(S)$ for the minimal dimension of a non-trivial projective representation for $S$ over a field of characteristic $p$. Then $|\Out(S)|\le 6\log_2{n}$. Furthermore, if $S$ is not an alternating group, and is not one of the groups $M_{12},M_{22},M_{24},J_2,\mathrm{Suz},Co_1,Co_2,\Lg_2(7),\Lg_2(8),\Lg_2(9),\Lg_3(4),\Ug_4(2),\Ug_4(3),$\\ $\PSp_6(2),\PSp_6(3)$, or $\POmega^+_8(2)$, then $|S|\le 2^{2(\log_2{n})^2+\log_2{n}}$.
\end{lemma}
\begin{proof}
Lower bounds for the values of $R_p(S)$ are given in \cite{LandSeitz}. The upper bounds in the lemma can then be deduced by inspection of these values and the orders of the outer automorphism groups of the finite simple groups.
\end{proof}
Notice that $|\Out(S)|= 6\log_2{R_p(S)}$ for $S=\Lg_3(4)$, so the first part of Lemma \ref{lem:RepSimple} is best possible.

We are now ready to prove Theorem \ref{thm:MainTheoremr}.
\begin{proof}[Proof of Theorem \ref{thm:MainTheoremr}] Let $X$ and $Y$ be subgroups of $\GL_n(\mathbb{F})$ with $X/Y\cong S$. We need to prove that $|X/Y|\le 2^{2(\log_2{n})^2+\log_2{n}}$.
We will first reduce to the case where $Y\le \Frat(X)$ and $X$ acts irreducibly on the natural $\GL_n(\mathbb{F})$-module $V$. 
Indeed, let $G$ be a subgroup of $X$ which is minimal with the property that $GY=X$. Then $G\cap Y\le \Frat(G)$, and $G/G\cap Y\cong X/Y$. Thus, by replacing $X$ by $G$, we may assume that $Y\le \Frat(X)$. Note also that since $O_p(X)\le Y$ and $X/O_p(X)$ embeds as a completely reducible subgroup of $\GL_n(\mathbb{F})$, we may assume that $O_p(X)=1$ and that $X$ acts completely reducibly on $V$. Next, let $W$ be an irreducible $X$-submodule of $V$ on which $X$ acts non-trivially, and write $X^W$ for the image of the induced action of $X$ on $W$. Let $\mathrm{Ker}_X(W)$ be the kernel of the action of $X$ on $W$. If $Y^W=X^W$, then $X=Y\mathrm{Ker}_X(W)$. But since $Y\le \Frat(X)$, this implies that $\mathrm{Ker}_X(W)=X$, contradicting our choice of $W$. Thus, we must have $Y^W\neq X^W$, and hence $X^W/Y^W\cong X/Y$, since $X^W/Y^W$ is a factor group of $X/Y$. Therefore, we may replace $X$ with $X^W$ and assume that $X$ is irreducible, as claimed.   

So assume that $X$ acts irreducibly on $V$, and suppose first that this action is imprimitive. Let $V=V_1\oplus\hdots\oplus V_t$ be an imprimitivity decomposition of $V$, and set $\Sigma:=\{V_1,\hdots,V_t\}$. Let $K:=\mathrm{Ker}_X(\Sigma)$ be the kernel of the action of $X$ on $\Sigma$. Then since $Y\le \Frat(X)$ we have (as above) $YK\neq X$. Thus, $X/Y\cong X^{\Sigma}/Y^{\Sigma}\le \Sym_t$. 
It follows from Lemma \ref{lem:PermSimple} that $|X/Y|$ is bounded above by $P(X/Y)^{2\log_2{P(X/Y)}}$. Hence, since $n\geq t\geq P(X/Y)$, we get $|X/Y|\le n^{2\log_{2}{n}}= 2^{2(\log{n})^2}$. This gives us what we need.  

Thus, we may assume that $X$ acts primitively on $V$. Then Corollary \ref{cor:PrimSimpleCor} implies that one of the following holds:
\begin{enumerate}[\upshape(i)]
    \item There exists a prime power divisor $r^m$ of $n$ such that $X/Y$ is a section of $\Sp_{2m}(r)$; or
    \item There exists a divisor $s^t$ of $n$, with $s\geq 2$, such that either:
    \begin{enumerate}[\upshape(a)]
        \item $X/Y$ has a projective irreducible representation of dimension $s$ over $\ol{\mathbb{F}}$; or
        \item $X/Y$ is a section of $\Sym_t$.
    \end{enumerate}
\end{enumerate}
If case (i) holds, then $|X/Y|\le |\Sp_{2m}(r)|\le r^{m(2m+1)}\le n^{2\log_2{n}+1}=2^{2(\log_2{n})^2+\log_2{n}}$; while
if (ii)(b) holds, then the same argument as in the imprimitive case above shows that $|X/Y|\le 2^{(\log_2{n})^2}$. So assume that case (ii)(a) holds. Then $n\geq R_{p}(X/Y)$. The result then follows from Lemma \ref{lem:RepSimple}.
\end{proof}

\section{The proof of Theorem \ref{thm:MainTheoremPrimLin}}\label{sec:PrimLin}
In this section, we prove Theorems \ref{thm:MainTheoremPrimLin}. With the function $M_p$ continuing to serve as our main tool, our approach to proving Theorem \ref{thm:MainTheoremPrimLin} will be to derive upper bounds on $M_p(G)$, for the primitive linear group $G$ in question. From the statement of the bound in the theorem, and with Corollary \ref{cor:PrimMpsCor} in mind, it will come as no surprise to the reader that the bulk of our work will be done in the groups $G_a/C_{G_a}(X)$ for $X\in\mathcal{X}$, where $G_a$ is as in Lemma \ref{lem:qlin} and $\mathcal{X}$ is as in Corollary \ref{cor:PrimMpsCor}.

We begin with those $X$ of the form $X=O_r(G_a)\not\le Z(G_a)$.
\begin{prop}\label{prop:basicI}
Let $q$ be a power of a prime $p$, and fix a prime divisor $r$ of $q-1$. Suppose that $H$ is a group of shape $N.L$, where $N\unlhd H$ is elementary abelian of order $r^{2m}$, and $L\le\Sp_{2m}(r)$ is completely reducible, acting naturally on $N$. Then $M_p(H)\le r^m(q^{r^m}-1)/(q-1)$, except in the following cases:
\begin{enumerate}[\upshape(i)]
   \item $(r,m,q)=(3,1,7)$ and $L=\Sp_2(3)$;
   \item $(r,m,q)=(2,1,5)$ and $L\in\{\Sp_2(2)',\Sp_2(2)\}$;
    \item $(r,m)=(2,1)$, $q\in\{7,11\}$, and $L=\Sp_2(2)$;
   \item $(r,m)=(2,2)$ and either:
   \begin{enumerate}[\upshape(a)]
       \item $q=3$ and $L\in\{C_5\rtimes C_4,\Sp_4(2)\}$;
       \item $q=5$ and $L=\Sp_2(2)\wr \Sym_2$;
       \item $q=7$ and $L\in\{\Alt_5(\text{two }\Sp_4(2)\text{-classes}),\Sp_4(2)',\Sp_4(2)\}$; or
       \item $q\in\{11,13\}$ and $L=\Sp_4(2)$.
   \end{enumerate}
\end{enumerate} 
\end{prop}
\begin{proof}
Note first that since $r$ divides $q-1$ and $|L|\le |\Sp_{2m}(r)|\le r^{2m^2+m}$, we have $M_p(H)\le |H|\le r^{2m^2+3m}$. Since $q>r\geq 2$, we readily see that this is at most $q^{r^m-1}$ if $m\geq 6$.
Moreover, for $1\le m\le 5$, the quantity $r^{1+2m}|\Sp_{2m}(r)|$ is bounded above by $q^{r^m-1}$ if $q>23$ (again, using the fact that $q>r$). Since $q^{r^m-1}<r^m(q^{r^m}-1)/(q-1)$, this gives us what we need in these cases. 

Suppose now that $m\le 5$, and that $q\le 23$. Since $r\mid q-1$, it is easy to check that $r^{1+2m}|\Sp_{2m}(r)|$ is bounded above by $q^{r^m-1}$ for all choices of $r,q,m$ (with the above restrictions), except when one of the following holds:
\begin{itemize}
    \item $(r,m,q)\in\{(3,1,4),(3,2,4),(3,1,7),(3,1,13)\}$;
    \item $(r,m)=(2,1)$ and $3\le q\le 23$ is odd;
    \item $(r,m)=(2,2)$ and $3\le q\le 19$ is odd;
    \item $(r,m)=(2,3)$ and $3\le q\le 13$ is odd;
    \item $(r,m)=(2,4)$ and $3\le q\le 7$ is odd; or
    \item $(r,m,q)=(2,5,3)$;
\end{itemize}
In each of the cases above, we can compute $M_p(L)$ for each completely reducible subgroup $L$ of $\Sp_{2m}(r)$. This gives us the precise value of $M_p(H)=M_p(N)M_p(L)=r^{2m}M_p(L)$. After doing this, we then see that the cases (i)--(iv) in the statement of the proposition are the only ones in which $M_p(H)>r^m(q^{r^m}-1)/(q-1)$. 
\end{proof}


We next move on to the analogous result for those $X$ from Corollary \ref{cor:PrimMpsCor} with $X$ of the form $X=T(S)$, for a quasisimple group $S$. It will come as no surprise to the reader that when computing $M_p(G_a/C_{G_a}(X))$, the cases $S\in\mathrm{Lie}(p)$ and $S\not\in\mathrm{Lie}(p)$ are substantially different. For this reason, we prove two separate propositions when dealing with these cases. 
Before distinguishing our arguments in this way, we prove the following lemma, which will be useful in both cases.
\begin{lemma}\label{lem:bothcases}
Let $q$ be a power of a prime $p$, and fix a finite simple group $S$. Suppose that $s$ is the dimension of an absolutely irreducible $\mathbb{F}_q[S]$-module $W$, and that $H$ is a subgroup of the wreath product $\Stab_{\Aut(S)}(W)\wr\Sym_t$ containing $S^t$.
\begin{enumerate}[\upshape(i)]
    \item There exists subgroups $A_i$ ($1\le i\le t$) of $\Stab_{\Aut(S)}(W)$ such that
    \begin{align*}
    M_p(H)\le M_p(S)^t\left(\prod_{i=1}^tM_p(A_i/S)\right)M_p(H/N)\le M_p(S)^t\left(\prod_{i=1}^tM_p(A_i/S)\right)\min\{t!,p^{t-1}\}.
\end{align*}
    \item Suppose that $M_p(A_i)\le q^{s-1}$ whenever $A_i$ is a subgroup of $\Stab_{\Aut(S)}(W)$ containing $S$. Then $M_p(H)\le q^{s^t-1}$.
    \item Suppose that $M_p(A_i)\le (q^{s}-1)/(q-1)$ whenever $A_i$ is a subgroup of $\Stab_{\Aut(S)}(W)$ containing $S$. Then $M_p(H)\le (q^{s^t}-1)/(q-1)$. 
\end{enumerate}
\end{lemma}
\begin{proof}
For ease of notation, set $n:=s^t$ and $A:=\Stab_{\Aut(S)}(W)$. We note first that by the hypothesis of the lemma, the intersection $N:=H\cap A^t$ of $H$ with the base group of $A\wr\Sym_t$ is a subdirect subgroup in a direct product $A_1\times\hdots\times A_t$, where each $A_i$ is a subgroup of $A$ containing $S$. It then follows from Corollary \ref{cor:MpCor} that
\begin{align}\label{lab:Stbound}
    M_p(H) &\le\left(\prod_{i=1}^tM_p(A_i)\right)M_p(H/N)\\ \nonumber  
    &\le M_p(S)^t\left(\prod_{i=1}^tM_p(A_i/S)\right)M_p(H/N).
\end{align}
Note also that $M_p(H/N)\le h_p(n)\le p^{t-1}$ by Corollary \ref{cor:MainTheoremPermCor}. Since $|H/N|\le t!$, part (i) follows.

Part (ii) follows from the first upper bound in (\ref{lab:Stbound}) and the fact that $(q^{s-1})^tq^{t-1}=q^{st-1}\le q^{{s^t}-1}$. So we just need to prove (iii). It is a routine exercise to check that if $q$ is a prime power and  $s\geq 2$, $t\geq 1$ are positive integers, then 
\begin{align*}
   \left(\dfrac{q^s-1}{q-1}\right)^tq^{t-1}\le \dfrac{q^{s^t}-1}{q-1}
\end{align*}
unless either $s=t=2$ or $q=s=2$ and $t=3$. If we are not in these exceptional cases, then the result follows from (\ref{lab:Stbound}). Suppose now that $s=t=2$. Then since $M_p(\Sym_2)=(2,p-1)$, the first bound in (\ref{lab:Stbound}) gives $M_p(H)\le (2,p-1)M_p(A_1)M_p(A_2)\le (2,p-1)(q^4-2q^2+1)/(q-1)^2$. It is clear that this is at most $(q^4-1)/(q-1)$ for all choices of $q$. The argument in the case $q=s=2$ and $t=3$ is entirely similar.    
\end{proof}

We can now derive upper bounds on $M_p(G_a/C_{G_a}(X))$ for those $X$ from Corollary \ref{cor:PrimMpsCor} with $X$ of the form $X=T(S)$, for a quasisimple group $S\not\in\mathrm{Lie}(p)$.
\begin{prop}\label{prop:basicII}
Let $q$ be a power of a prime $p$, and fix a finite simple group $S\not\in\mathrm{Lie}(p)$. Suppose that $s$ is the dimension of an absolutely irreducible $\mathbb{F}_q[S]$-module $W$, and that $H$ is a subgroup of the wreath product $\Stab_{\Aut(S)}(W)\wr\Sym_t$ containing $S^t$. Then $M_p(H)\le (s^t)_{p'}(q^{s^t}-1)/(q-1)$, except in the following cases:
\begin{enumerate}[\upshape(i)]
    \item $H=S=\Alt_5$, $(s,t)=(2,1)$, and $q\in\{11,19\}$; or 
    \item $H=\Alt_5\wr\Sym_2$, $(s,t)=(2,2)$, and $q=11$.
\end{enumerate}
\end{prop}
\begin{proof}
For ease of notation, set $n:=s^t$ and $A:=\Stab_{\Aut(S)}(W)$. We will in fact prove that $M_p(H)\le q^{s^t-1}$, except in the following cases:
\begin{enumerate}[\upshape(1)]
    \item $H=S=\Alt_5$, $(s,t)=(2,1)$, and $q\in\{11,19,29,31,41,59\}$;
    \item $S=\Alt_5$, $(s,t)=(2,2)$, and $q\in\{11,19\}$; or
    \item $H=S=\Lg_2(7)$, $(s,t)=(3,1)$, and $q=11$.
\end{enumerate}
Since $q^{s^t-1}<(s^t)_{p'}(q^{s^t}-1)/(q-1)$, this will give us what we need in all but the three cases listed above. It is then easy to check that $M_p(H)\le (s^t)_{p'}(q^{s^t}-1)/(q-1)$ in these cases, except when either $H=S=\Alt_5$, $(s,t)=(2,1)$, and $q\in\{11,19\}$; or $H=\Alt_5\wr\Sym_2$, $(s,t)=(2,2)$, and $q=11$.

Suppose first that $s=2$. It is well-known that the only nonabelian finite simple group $S\not\in\mathrm{Lie}(p)$ with a projective irreducible representation in characteristic $p$ is $\Alt_5$, and in this case, $A=S$ (one can see this from \cite[Table 8.2]{BHRD}, for example). Moreover, $p\not\in\{2,5\}$, since $\Alt_5\cong \Lg_2(4)\cong\Lg_2(5)$. Since $\GL_2(7)$ has no subgroup isomorphic to a cover of $\Alt_5$, we deduce that $q\geq 9$. Assume that $t=1$, so that $H=S$. If $p=3$, then we quickly compute that $M_3(H)=M_3(S)=3<q=q^{n-1}$. Suppose now that $p\geq 7$. Then $M_p(H)=M_p(S)=|S|$, so we get $M_p(H)\le q=q^{n-1}$ if and only if $q\geq |S|=60$. Using \cite[Table 2]{BHRD}, we see that if $p\geq 7$ and $q\le 60$, then $\Alt_5$ has a projective representation over $\mathbb{F}_{q}$ if and only if $q\in\{11,19,29,31,41,59\}$. This is the case listed (1) in the first paragraph above.

Continuing to assume that $s=2$, suppose next that $t>1$. Assume first that either $p=3$ or $q\geq 60$. Then $M_p(H)\le q^{2^t-1}=q^{s^t-1}$ by Lemma \ref{lem:bothcases}(ii).
Assume now that $7<q<60$ and $p\geq 7$. Then it follows from Lemma \ref{lem:bothcases}(i) that $M_p(H)\le 60^t\min\{p^{t-1},t!\}$. If $t\geq 3$, then this is less than or equal to $q^{2^t-1}=q^{n-1}$ for all of choices of $7<q<60$ (with $p\geq 7$). So assume that $t=2$. Then $M_p(H)\le 60^t2$ as above, and this is less than or equal to $q^{4-1}=q^{n-1}$ unless $q\in\{11,13,17,19\}$. However, we can see from \cite[Table 2]{BHRD} that among these values of $q$, $\Alt_5$ has a projective representation over $\mathbb{F}_{11}$ and $\mathbb{F}_{19}$ only. This is the exception numbered (2) in the first paragraph above.

We can analyse the case $s=3$ in an entirely similar way: From \cite{HM} we see that $S\in\{\Alt_6,\Alt_7\text{($q=5$)},\Lg_2(7)\}$. Furthermore, using \cite[Tables 8.3--8.6]{BHRD}, we have $A=S$ unless $S=\Alt_6$ and $q=5$, in which case $|A/S|\le 2$. Arguing as in the paragraphs above then gives the bound $M_p(H)\le q^{n-1}$, apart from when $t=1$, $S=\Lg_2(7)$, and $q=11$. 

We may assume, therefore, that $s>3$. By Lemma \ref{lem:bothcases}(ii), it suffices to complete the proof of the proposition in the case $t=1$. Thus we may assume, for the remainder of the proof, that $t=1$ (whence $H$ is almost simple). We need to prove that $M_p(H)\le q^{s-1}$. To this end, suppose first that if $S$ is an alternating group, then the degree is at most $8$. We will assume first that $s\geq 103$. Then one can quickly check that if either $S\in\{\Alt_m\text{ : }5\le m\le 27\}$, or if $S$ is one of the exceptions listed in Lemma \ref{lem:RepSimple}, then $|S|\le 2^{(\log_2{s})^2+\log_2{s}}$. Thus, $|S|\le 2^{2(\log_2{s})^2+\log_2{s}}$ in all of the cases under consideration, by Theorem \ref{thm:MainTheoremr}. Furthermore, $|H/S|\le |\Out(S)|\le 6\log_2{s}$ by Lemma \ref{lem:RepSimple}. It follows from Lemma \ref{lem:bothcases}(i) that $M_p(H)\le 2^{2(\log_2{s})^2+\log_2{s}}(6\log_2{s})$. Since $s\geq 103$, one can check that this is bounded above by $q^{s^t-1}=q^{n-1}$ for all values of $q\geq 2$. 

Continuing to assume that $S$ is not an alternating group of degree greater than $27$, suppose next that $4\le s\le 102$. The possibilities for the simple group $S$ are then given by \cite[Tables 2 and 3]{HM}. We can then simply go through each of these possibilities for $S$, and compute $|\Aut(S)|$. We also define the set $\pi:=\pi_S$ and the prime $p_0:=p_0(S)$ as follows: let $\pi$ be the set of primes for which $S\in\mathrm{Lie}(r_i)$. Then choose $p_0$ to be the smallest prime with $p_0\not\in \{r_i\}_i$. For example, if $S=\Lg_2(4)=\Lg_2(5)$, then $\pi=\{2,4\}$ and $p_0=3$. We then have $q\geq p_0$, since we are assuming that $S\not\in\mathrm{Lie}(p)$. Thus, if $|\Aut(S)|\le p_0^{s-1}$, then we have $M_p(H)\le |\Aut(S)|\le p_0^{s-1}\le q^{s-1}$. The groups $S$ which do not satisfy $|\Aut(S/Z(S))|\le p_0^{s-1}$, for some value of $s$, are as follows:
\begin{align}\label{lab:EX00}
&\Lg_2(q_1)\text{($q_1\le 31,q_1\neq 16$),}\PSp_4(q_1)\text{($q_1\le 7$),}\Lg_3(3),\PSp_6(3),\POmega_7(3),\Ug_4(2),\Ug_4(3),\\ \nonumber
&\Ug_5(2),\PSp_8(2),\PSp_8(3),G_2(2)'\text{($q=5$)},G_2(3),M_{11},M_{12},M_{22},M_{23},M_{24},J_1,J_2,\\ \nonumber
& J_3,\mathrm{Suz},HS,McL,Co_1,Co_2,Co_3,Fi_{22},Ru,\Alt_m\text{($5\le m\le 8$)}.  
\end{align}
The arguments to prove that $M_p(H)\le q^{n-1}$ are almost identical for each of the groups listed at (\ref{lab:EX00}), so we will just give the details for the $\Lg_2(q_1)$ cases.
So assume that $S\cong \Lg_2(q_1)$, for some $q_1\le 31$ with $q_1\neq 16$. Suppose first that $S=\Lg_2(4)=\Lg_2(5)\cong\Alt_5$. Then $\pi_S=\{2,5\}$, and we have $M_p(H)\le |\Aut(S)|=120\le q^{4-1}\le q^{s-1}$ if $q>3$. If $q=3$, then we calculate $M_3(S)=3$, so $M_3(H)\le 2M_3(S)\le 6<3^{4-1}<3^{s-1}$. The result follows. Suppose now $q_1\not\in\{4,5\}$. Then $n=s\geq \max\{4,R_p(S)\}$. By \cite{LandSeitz}, it follows that $n=s\geq 4$ if $q_1=9$, and $n=s\geq \max\{4,(q_1-1)/(q_1-1,2)\}$ otherwise. We then see that the inequality $|\Aut(S)|\le q^{n-1}$ holds for $q\geq q'$, where $q'$ is the value associated to $\Lg_2(q_1)$ in the second column of Table \ref{tab:L2Ex00}.

\begin{center}
\begin{tabular}{c|c}
  $S$   &  $q'$\\
  \hline
  $\Lg_2(7)$      &  $7$\\
  $\Lg_2(8)$      &  $4$\\
  $\Lg_2(9)$      &  $13$\\
  $\Lg_2(11)$      & $7$\\
  \hline
\end{tabular}
\quad
\begin{tabular}{c|c}
  $S$   &  $q'$\\
  \hline
  $\Lg_2(13)$      &  $5$\\
  $\Lg_2(17)$      &  $4$\\
  $\Lg_2(19)$      &  $4$\\
  $\Lg_2(23)$      &  $3$\\
  \hline
\end{tabular}
\quad
\begin{tabular}{c|c}
  $S$   &  $q'$\\
  \hline
  $\Lg_2(25)$      &  $3$\\
  $\Lg_2(27)$      &  $3$\\
  $\Lg_2(29)$      &  $3$\\
  $\Lg_2(31)$      &  $3$\\
  \hline
\end{tabular}
\begin{table}[ht!]
    \caption{$|\Aut(S)|\le q^{n-1}$ for $q\geq q'$ as above.}
    \label{tab:L2Ex00}
\end{table}
\end{center}

Thus, we just need to prove that $M_p(X)\le q^{n-1}$ whenever $q<q'$ and $p\not\in\pi_S$. We do this by computing the precise value of $M_p(S)$ in each case, and then applying the bound $M_p(H)\le M_p(S)M_p(\Out(S))$. For example, when $S=\Lg_2(11)$, we need to check that $M_p(H)\le q^{s-1}$ when $q\in\{2,3,4,5\}$. We have $M_2(S)=12$; $M_3(S)=6$; and $M_5(S)=5$. Also, $s\geq R_p(S)=(11-1)/2=5$. The computations above then give $M_p(S)\le q^{5-1}\le q^{n-1}$ in each case. This argument works in all cases, except when either $S=\Lg_2(7)$ and $q=5$; or $S=\Lg_2(9)\cong\Alt_6$ and $q\in\{7,11\}$. In each of these cases, we can use the Atlas of Brauer Characters \cite{BrAtlas} to check that since $s\geq 4$, we must in fact have $s\geq 5$. We then see that $M_p(H)\le |\Aut(S)|\le q^{5-1}$ in each case.

Finally, assume that $S$ is an alternating group of degree $m\geq 9$. Then $R_p(H)=m-2$ if $p\mid m$, and $R_p(H)=m-1$ otherwise (see \cite[Proposition 5.3.7]{KL}, for example). If $p\geq m$, then it follows that $M_p(H)\le |\Sym_m|\le 6p^{m-3}\le p^{m-2}\le q^{n-1}$. If $p<m$, then Theorem \ref{thm:AltThm} yields $M_p(H)\le 2q^{m/2}$. Since $m\geq 9$ implies that $m/2\le m-4$, we deduce that  $M_p(H)\le 2q^{m-4}\le q^{m-3}\le q^{n-1}$. This gives us what we need, and completes the proof.
\end{proof}

Continuing with our analysis of the function $M_p$ and its values on those $G_a/C_{G_a}(X)$ with $G_a$, $X$ as in Corollary \ref{cor:PrimMpsCor}, we next deal with the case $X=T(S)$ with $S\in\mathrm{Lie}(p)$.
\begin{prop}\label{prop:basicIIp}
Let $q=p^f$ be a power of a prime $p$, and fix a finite simple group $S\in\mathrm{Lie}(p)$. Write $\hat{S}$ for the Schur cover of $S$. Suppose that $s$ is the dimension of an absolutely irreducible $\mathbb{F}_q[\hat{S}]$-module $W$, and that $H$ is a subgroup of the wreath product $\Stab_{\Aut(S)}(W)\wr\Sym_t$ containing $S^t$. Suppose also that $W$ is not realizable over any proper subfield of $\mathbb{F}_q$. Then $M_p(H)\le (q^{s^t}-1)/(q-1)$, except in the following cases:
\begin{enumerate}[\upshape(i)]
    \item $S=\Lg_n(q)$, $t=1$, and $W$ is quasiequivalent to the natural module for $\hat{S}$. In this case, $M_p(H)\le q^{s-1}s_{p'}$.
    \item $S=\Lg_2(5)$, $(s,t,q)=(2,2,5)$, and $H=\mathrm{PGL}_2(5)\wr\Sym_2$. In this case, $M_p(H)=200<4(5^4-1)/(5-1)$.
\end{enumerate}
\end{prop}
\begin{proof}
As in the proof of Proposition \ref{prop:basicII}, we will set $n:=s^t$ and $A:=\Stab_{\Aut(S)}(W)$. We will also adopt similar notation to that set out in the comments preceding the statement of Theorem \ref{thm:MainTheoremp}. Indeed, write $S=O^{p'}(\ol{X}_{\ol{\sigma}})$ where $\ol{X}$ is a simple algebraic group and $\ol{\sigma}:\ol{X}\rightarrow\ol{X}$ is a Steinberg endomorphism of $\ol{S}$. As before, we will also write $\ol{\sigma}=\ol{\sigma_p}^e\ol{\gamma}^i$, where $\ol{\sigma_p}$ is a Frobenius automorphism of $\ol{X}$, and either 
\begin{enumerate}[(a)]
    \item $\ol{X}\in\{A_{\ell},D_{\ell},E_6\}$ and $\ol{\gamma}$ is a graph automorphism of $\ol{X}$ as in \cite[Theorem 1.15.2(a)]{GLS};
    \item $\ol{X}\in\{B_2\text{ }(p=2),G_2\text{ }(p=3)\text{, }F_4\text{ }(p=2)\}\}$, and $\ol{\gamma}$ is as in \cite[Theorem 1.15.4(b)]{GLS}; or
    \item $\ol{X}$ is not one of the groups in (a) or (b) above, and $\ol{\gamma}=1$.
\end{enumerate}
Writing $q_1:=p^e$ for the level of $S$; $\ell$ for the untwisted Lie rank of $S$; and $d$ for the order of $\gamma^i$, we will also shorten our notation and write $S={}^dS_{\ell}(q_1)$. Recall that $q=p^f$, and that $\Inndiag(S)$ is defined as $\Inndiag(S):=\ol{X}_{\ol{\sigma}}$.

We first discuss the structure of the group $A$. It is well-known that $\Inndiag(S)$ $\le A$ (see for example \cite[Proposition 5.1.9(i)]{BHRD}). We now determine the group $A/\Inndiag(S)$. The information we require is in \cite[Proposition 5.4.6]{KL} (and the comments afterwards): Recalling that $d\in\{1,2,3\}$, \cite[Proposition 5.4.6]{KL} implies that there exists an irreducible $\mathbb{F}_q[\hat{S}]$-module $M$ such that one of the following holds:
\begin{enumerate}
    \item $f\mid e$ and $s=\dim{W}=(\dim{M})^{e/f}$.
    \item $d\in\{2,3\}$, $f\nmid e$, $f\mid de$, and $s=\dim{W}=(\dim{M})^{de/f}$.
\end{enumerate}
Moreover (see \cite[Proposition 5.1.9(i)]{BHRD} and \cite[Propositions 5.4.2, 5.4.3 and 5.4.6]{KL}), if we are in case (1) then $|A/\Inndiag(S)|=g(e/f)$, where $g:=2$ if $W$ is self-dual, and $g:=1$ otherwise. If we are in case (2), then $|A/\Inndiag(S)|=g(e/f)$ where $g\le d$. 

We are now ready to prove the proposition. Assume first that $S=A_{\ell}(q_1)$ and $t=1$. Then $\dim{M}\geq R_p(S)=\ell+1$ by \cite[Proposition 5.4.13]{KL}. Also, $M_p(S)\le q_1^{\ell}$, by Theorem \ref{thm:MainTheoremp}. It follows from the paragraph above and Corollary \ref{cor:MpCor} that $M_p(H)\le q_1^{\ell}(q_1-1,\ell+1)2(e/f)\le q_1^{\ell}(\ell+1)2(e/f)$. Now, assume that $W$ is not quasiequivalent to the natural $(\ell+1)$-dimensional module for $\hat{S}$ over $\mathbb{F}_{q_1}$. Then either $e/f>1$; or $e/f=1$ and $\dim{M}>\ell+1$. Suppose first that $e/f>1$. Then $M_p(H)\le q_1^{\ell}(\ell+1)2(e/f)$ as above, and this is less than $(q^{(\ell+1)^{e/f}}-1)/(q-1)\le (q^s-1)/(q-1)$ as long as $(q,e/f,\ell)\neq (2,2,1)$. If $(q,e/f,\ell)=(2,2,1)$, then $S=\Lg_2(4)\cong\Alt_5$. Since $p=2$ and $M_2(\Alt_5)=4$, we get $M_p(H)=M_p(S)=4$ in this case, and this gives us what we need. Suppose next that $e/f=1$ and $\dim{M}>\ell+1$. Note that $q_1=q$ in this case, so $q\geq 4$ if $\ell=1$. Moreover, \cite[Proposition 5.4.11]{KL} implies that one of the following holds:
\begin{enumerate}
    \item $\ell>1$ and $\dim{M}\geq \ell(\ell+1)/2$; or
    \item $\ell=1$ and $\dim{M}\geq (\ell+1)(\ell+2)/2=3$.
\end{enumerate}
Assume first that case (1) holds. Then since $q_1^{\ell}(\ell+1)2=q^{\ell}(\ell+1)2$ is less than $(q^{\frac{1}{2}\ell(\ell+1)}-1)/(q-1)\le (q^s-1)/(q-1)$ when either $q\geq 5$ or $\ell\geq 3$, the bound $M_p(H)\le q_1^{\ell}(\ell+1)2$ from above gives us what we need in these cases. Suppose now that $\ell=2$ and $q\in\{2,3,4\}$, then the bound $M_p(H)\le q^{\ell}(q-1,\ell+1)2=2$ gives the result in each case. Assume next that case (2) holds. Then $q\geq 4$, since $S$ is nonabelian simple. Also, $M_p(H)\le 2M_p(S)\le 2q$, by Theorem \ref{thm:MainTheoremp}. Since $2q\le (q^3-1)/(q-1)\le (q^s-1)/(q-1)$ when $q\geq 4$, the result again follows.     

Thus, in the case $S=A_{\ell}(q_1)$ and $W$ is not quasiequivalent to the natural $\mathbb{F}_{q_1}[\hat{S}]$-module, we have proved that $M_p(H)\le (q^s-1)/(q-1)$ if $t=1$. By Lemma \ref{lem:bothcases}(iii), it follows that $M_p(H)\le (q^{s^t}-1)/(q-1)$ for all $t\geq 1$ in this case

Suppose next that $S=A_{\ell}(q_1)$, but that $W$ \emph{is} quasiequivalent to the natural $\mathbb{F}_{q_1}[\hat{S}]$-module. Then we have $q_1=q$ and $A=\Inndiag(S)$. Thus, $M_p(A_i)\le M_p(S)M_p(A_i/S)\le q^{\ell}(q-1,\ell+1)\le q^{\ell}(\ell+1)$ whenever $S\le A_i\le A$, using Theorem \ref{thm:MainTheoremp}. It follows in particular that if $t=1$, then $M_p(H)\le q^{\ell}(\ell+1)_{p'}=q^{s-1}s_{p'}$. This is case (i) in the statement of the proposition. 

We may assume, therefore, that $t>1$. Suppose first that $(\ell,t)\neq (1,2)$. By Lemma \ref{lem:bothcases}(i), we have $M_p(H)\le q^{\ell t}(\ell+1)^tq^{t-1}=q^{t(\ell+1)-1}(\ell+1)^t$. Since this is less than $(q^{(\ell+1)^t}-1)/(q-1)=(q^{s^t}-1)/(q-1)$ for all choices of $q$, $\ell$, and $t\geq 2$ with $(\ell,t)\neq (1,2)$, the result follows in this case. Assume now that $(\ell,t)=(1,2)$. Then by Lemma \ref{lem:bothcases}(i), we have $M_p(H)\le q^{2\ell}(\ell+1)^22=8q^2$. This is less than $(q^{(\ell+1)^2}-1)/(q-1)=(q^4-1)/(q-1)$ if $q\geq 7$. So assume that $q\le 5$. Then $q\in\{4,5\}$, since $S$ is nonabelian simple. It follows that $S\cong \Alt_5$, and $M_2(S)=4$, $M_5(\Alt_5)=5$. Thus, if $q=4$ then we have $M_p(H)=M_p(S)\le 4^2=16<(4^4-1)/(4-1)$. If $q=5$, then $H$ has shape $\Alt_5^2.K$, where $K$ is a $2$-group of order at most $8$. It follows that $M_p(H)=25\times|K|$. We then get $M_p(H)\le (5^4-1)/(5-1)$ if and only if $|K|<8$. Thus, $M_p(H)>(5^4-1)/(5-1)$ if and only if $H=\GL_2(5)\wr \Sym_2\le \GL_4(5)$, and this gives us what we need.

This completes the proof of the proposition in the case $S=A_{\ell}(q)$. We may assume therefore, for the remainder of the proof, that $S\neq A_{\ell}(q)$. By Lemma \ref{lem:bothcases}(iii), we may also assume that $t=1$ (whence $H$ is almost simple). Suppose first that $S\neq {}^2A_{\ell}(q_1)$, and let $m:=\dim{M}$, so that $s\geq m^{e/f}$. If $S=\PSp_4(2)'\cong\Alt_6$, then $M_2(S)=4$, $R_p(S)=3$ (see \cite[Proposition 5.4.13]{KL}), and $\Aut(S)/S$ is a $2$-group. Thus, $M_2(H)=M_2(S)\le (q^3-1)/(q-1)\le (q^s-1)/(q-1)$ as needed. So we may assume that $S\neq \PSp_4(2)'$.
Then inspecting the upper bounds for $M_p(S)$ given by Theorem \ref{thm:MainTheoremp}, and the values of $R_p(S)$ given by \cite[Proposition 5.4.13]{KL}, we quickly see that $M_p(S)\le 2q_1^{m/2}$. Since $|\Inndiag(S)/S|\le 4$ in these cases, it follows (see the second paragraph above) that $M_p(H)\le 8q_1^{m/2}(e/f)3=24q_1^{m/2}(e/f)=24q^{em/2f}(e/f)$. This is less than $(q^{m^{e/f}}-1)/(q-1)\le (q^{s}-1)/(q-1)$ if either $e/f\geq 4$ or $m\geq 10$. So assume that $e/f\in\{1,2,3\}$ and $m=\dim{M}\le 9$. We can then do more specific analysis according to the isomorphism class of $S$. For example, if $S=C_{\ell}(q)$ (with $\ell\geq 2$), then $m\le 9$ implies that $\ell\le 4$, by \cite[Proposition 5.4.13]{KL}. Also, $|A/S|$ divides $(2,p-1)(e/f)(2,p)$ by the second paragraph above, and $M_p(S)\le (2,p)q_1^{m/2}=(2,p)q^{em/2f}$ by Theorem \ref{thm:MainTheoremp}. Thus, in either of the cases $p=2$ or $p\neq 2$, we get $M_p(H)\le 2q^{em/2f}(e/f)$, by definition of the function $M_p$. Since $m\geq 2\ell\geq 4$ and $2q^{em/2f}(e/f)\le (q^{m^{e/f}}-1)/(q-1)$ for $m\geq 4$, the result follows. The arguments for dealing with the remaining groups $S$, with $S\neq {}^2A_{\ell}(q_1)$, are entirely similar.

Assume finally that $S={}^2A_{\ell}(q_1)$ with $\ell\geq 2$. Then $M_p(S)\le q_1^{\ell}=q^{e\ell/f}$ by Theorem \ref{thm:MainTheoremp}. By arguing as in the second paragraph above, we see that $M_p(H)\le q^{e\ell/f}(q_1+1,\ell+1)2(e/f)\le q^{e\ell/f}(\ell+1)2(e/f)$. If $e/f>1$, then this is less than $(q^{(\ell+1)^{e/f}}-1)/(q-1)\le (q^s-1)/(q-1)$, so assume that $e=f$. Then by \cite[Proposition 5.4.6]{KL}, the module $M$ is self-dual. Also, $|A/\Inndiag(S)|=2$ (see the second paragraph above). On the other hand, since $\ell\geq 2$, the natural $\mathbb{F}_q[\hat{S}]$-module is not self-dual (see \cite[Proposition 5.4.2 and 5.4.3]{KL}). Thus, by \cite[Proposition 5.4.11]{KL} we have $s=\dim{W}=\dim{M}\geq \frac{1}{2}\ell(\ell+1)$. Arguing as above, and since $|A/\Inndiag(S)|=2$, we deduce that $M_p(H)\le q^{\ell}(q+1,\ell+1)2\le 2q^{\ell}(\ell+1)$ if $p>2$, and $M_p(H)\le q^{\ell}(q+1,\ell+1)\le q^{\ell}(\ell+1)$ if $p=2$. These upper bounds give $M_p(H)\le (q^{\frac{1}{2}\ell(\ell+1)}-1)/(q-1)\le (q^s-1)/(q-1)$ if $\ell>2$. So assume that $\ell=2$. Then $s\geq 6$ by \cite[Theorem 4.4]{Lubeck}, and the result follows as above.  
\end{proof}

Before proceeding to the proof of Theorem \ref{thm:MainTheoremPrimLin}, it will be useful for us to state the following consequence of Propositions \ref{prop:basicI}, \ref{prop:basicII}, and \ref{prop:basicIIp}.
\begin{cor}\label{cor:MainBasic}
Let $q$ be a power of a prime $p$, and suppose that $H$ is a finite group satisfying one of the following:
\begin{enumerate}[(a)]
    \item $H$ has shape $N.L$, where $N\unlhd H$ is elementary abelian of order $r^{2m}$, and $L\le\Sp_{2m}(r)$ is completely reducible, acting naturally on $N$; or
    \item there exists a finite simple group $S$ with an $s$-dimensional absolutely irreducible projective $\mathbb{F}_q[S]$-module $W$, such that $H$ is a subgroup of the wreath product $\Stab_{\Aut(S)}(W)\wr\Sym_t$ containing $S^t$.
\end{enumerate}
In case (a), set $d:=r^m$, and in case (b) set $d:=s^t$. Then $M_p(H)\le d_{p'}(q^{d}-1)/(q-1)$, unless $H$ lies in Table \ref{tab:EXX}. In these latter cases, the value of $M_p(H)$ is given in the fourth column of Table \ref{tab:EXX}.
\end{cor}

\begin{table}[]
    \centering
    \begin{tabular}{c|c|c|c}
     $H$    & $d$ & $q$ & $M_p(H)$\\
     \hline
     $3^2.\Sp_2(3)$ & $3$    & $7$ & $216$\\
     $2^2.\Sp_2(2)$ & $2$    & $5,7$ & $24$\\
     $2^4.\Sp_4(2)$ & $4$    & $3$ & $288$\\
     $2^4.\Sp_4(2)$ & $4$    & $7,11,13$ & $11520$\\
     $2^4.(5\rtimes 4)$ & $4$    & $3$ & $320$\\
     $2^4.(\Sp_2(2)\wr \Sym_2)$ & $4$    & $5$ & $1152$\\ 
     $2^4.\Sp_4(2)'$ & $4$    & $7$ & $5760$\\ 
     $2^4.\Sym_5$ & $4$    & $7$ & $1920$\\
     $\Alt_5$ & $2$    & $11,19$ & $60$\\
     $\Alt_5\wr\Sym_2$ & $4$    & $11$ & $7200$\\
     \hline
    \end{tabular}
    \caption{Exceptions from Corollary \ref{cor:MainBasic}}
    \label{tab:EXX}
\end{table}

We are almost ready to prove Theorem \ref{thm:MainTheoremPrimLin}. To facilitate a tidier proof of the theorem, we first prove the following easy lemma.
\begin{lemma}\label{lem:strange}
Let $q_1$ be a power of a prime $p$, and let $(M_1,d_1),\hdots,(M_k,d_k)$ be pairs of natural numbers, with $d_\geq 2$ for all $i$. Number the rows in Table \ref{tab:EXX} from $1$ to $10$. Suppose that for each $i$, the pair $(M_i,d_i)$ satisfies one of the following:
\begin{enumerate}[(i)]
    \item $M_i:= (d_i)_{p'}(q_1^{d_i}-1)/(q_1-1)$; or
    \item there exists $f(i)$ with $1\le f(i)\le 10$ such that $d_i,q_1$ and $M_i$ lie in the third, fourth and fifth column of Row $f(i)$ in Table \ref{tab:EXX}, respectively.
\end{enumerate}
Assume also that each of the following conditions hold:
\begin{itemize}
    \item if $(M_i,d_i)$ and $(M_j,d_j)$ satisfy (ii) above, with $i\neq j$ and $f(i)=f(j)$, then $f(i)\in\{9,10\}$.
    \item if $(M_i,d_i)$ and $(M_j,d_j)$ satisfy (ii) above with $i\neq j$ and $2\le f(i)\le 8$, then $f(j)\in\{1,9,10\}$.
    \item if $k=1$, then $(M_1,d_1)$ satisfies (i).
\end{itemize}
Then $(q_1-1)\prod_{i=1}^kM_i\le (n_1)_{p'}(q_1^{n_1}-1)$, where $n_1:=\prod_{i=1}^kd_i$.
\end{lemma}
\begin{proof}
We prove the claim by induction on $k$, with the case $k=1$ being trivial. So assume that $k>1$, and that the lemma holds for all smaller $k$. 

Suppose first that $k=2$, and that both $(M_1,d_1)$ and $(M_2,d_2)$ satisfy (ii) in the statement of the lemma. By assumption, the triples $(d_i,q_1,M_i)$ cannot both come from rows 2--8 in Table \ref{tab:EXX}. With this in mind, we can now prove that $(q_1-1)M_1M_2\le (n_1)_{p'}(q_1^{n_1}-1)$ via a quick calculation. For example, if $q_1=7$, then  $\{f(1),f(2)\}=\{1,2\},\{1,4\},\{1,7\}$ or $\{1,8\}$. It then follows, respectively, that $n_1=6$, $12$, $12$, or $12$; and that $(7-1)M_1M_2$ is at most $(7-1)\times 216\times 24$, $(7-1)\times 216\times 11520$, $(7-1)\times 216\times 5760$, or $(7-1)\times 216\times 1920$. This yields $(7-1)M_1M_2<6(7^6-1)$ in the first case; and $(7-1)M_1M_2<12(7^{12}-1)$ in the next three cases. This gives us what we need in each case. 
Thus, we may assume that if $k=2$, then at least one of the pairs $(M_i,d_i)$, say $(M_1,d_1)$, satisfies (i) in the statement of the lemma. The inductive hypothesis then implies that $(q_1-1)\prod_{i=1}^{k-1}M_i\le (d)_{p'}(q_1^{d}-1)$, where $d:=\prod_{i=1}^{k-1}d_i$. So it will suffice to prove that $(d)_{p'}(q_1^{d}-1)M_k\le (dd_k)_{p'}(q_1^{dd_k}-1)$. If $M_k\le (d_k)_{p'}(q_1^{d_k}-1)$, then this is clear, since each $d_i$ is at least $2$. So assume that there exists $f(k)$, with $1\le f(k)\le 10$, such that $d_k,q_1$ and $M_k$ lie in the third, fourth and fifth column of Row $f(k)$ in Table \ref{tab:EXX}, respectively. It is then just a matter of going through each of the possibilities for $(d_k,q_1,M_k)$ from Table \ref{tab:EXX}, and deriving the required upper bound. For example, if $f(k)=4$, so that $q_1\in\{7,11,13\}$, $d_k=4$, and $M_k=11520$, then $(d)_{p'}(q_1^{d}-1)M_k=11520(d)_{p'}(q_1^{d}-1)$. Note that $11520(q_1^{d}-1)$ is less than $4(q_1^{4d}-1)$ for all choices of $d\geq 2$, and $q_1\in\{7,11,13\}$. This gives us what we need. The remaining possibilities for $(d_k,q_1,M_k)$ from Table \ref{tab:EXX} can be dealt with in an entirely similar way.
\end{proof}

We are now ready to prove Theorem \ref{thm:MainTheoremPrimLin}.
\begin{proof}[Proof of Theorem \ref{thm:MainTheoremPrimLin}]
We will in fact prove that if $G\le \GL_n(q)$ is weakly quasiprimitive; $G$ is not the group $2^{1+4}.\Sp_4(2)\le \GL_4(3)$ (with $(n,q)=(4,3)$); and $G$ is not one of the exceptions listed in part (i) of the statement of the theorem, then
\begin{align}\label{lab:MainClaim}
M_p(G_a)\le (n/d)_{p'}(q^n-1),    
\end{align}
where $d$ and $G_a\le \GL_{n/d}(q^d)$ are as in Lemma \ref{lem:qlin}. Since $M_p(G)\le d_{p'}M_p(G_a)$ by Lemma \ref{lem:qlin} and the definition of the function $M_p$, the theorem will follow from (\ref{lab:MainClaim}) when $(G,n,q)$ is not $(2^{1+4}.\Sp_4(2),4,3)$. In this latter case, we can compute directly that the proportion of $3$-elements in $G$ is at least $1/[4(3^4-1)]$.


So assume that $G$ is not one of the exceptions listed in the statement of the theorem, and that $(G,n,q)$ is not $(2^{1+4}.\Sp_4(2),4,3)$. Let $d$ and $G_a$ be as above. Also, let $\mathcal{R}$, $\mathcal{S}$, $q_1=q^d$, $m_r$, $s(S)$, $t(S)$, and $\mathcal{X}$ be as defined in Lemmas \ref{lem:qlin} and \ref{lem:MainPrimLin}. For $X\in\mathcal{X}$, define $d_X:=r^{m_r}$ if $X=O_r(G_a)$ for some $r\in\mathcal{R}$; and define $d_X:=s(S)^{t(S)}$ if $X=T(S)$ for some $S\in\mathcal{S}$. By
Corollary \ref{cor:MainBasic}, either $M_p(G_a/C_{G_a}(X))\le (d_X)_{p'}(q_1^{d_X}-1)/(q_1-1)$, or $G_a/C_{G_a}(X)$ lies in Table \ref{tab:EXX}. In the former case, define $M_X:=(d_X)_{p'}(q_1^{d_X}-1)/(q_1-1)$. In the latter case, define $M_X$ to be the upper bound for $M_p(G_a/C_{G_a}(X))$ given in the fourth column of Table \ref{tab:EXX}. 

Suppose first that $|\mathcal{X}|=1$ and that $(G_a/C_{G_a}(X),d_X,q_1)$ lies in Table \ref{tab:EXX}, where $\mathcal{X}=\{X\}$. Since $q_1$ is prime in each of these cases, we must have $d=1$ and $G= G_a$. If $X=O_r(G)$ for some $r\in\mathcal{R}$, then it follows that $G$ has shape $Z\circ r^{1+2m}.(H/O_r(H))$, where $Z$ is a subgroup of $Z(\GL_n(q))$, and $H:=G/C_G(X)$. We can then read off the possibilities for $G$ from Table \ref{tab:EXX}, and these give us the exceptions listed in the first nine rows of Table \ref{tab:EXMainTheoremPrimLin}. We do the same thing in the case $X=T(S)$, with $S\in\mathcal{S}$: we see that $G$ is one of the groups in the final two rows of Table \ref{tab:EXMainTheoremPrimLin}. Note here that despite the presence of the group $2^{4}.\Sp_4(2)$ in Table \ref{tab:EXX}, the associated primitive group $G:=2^{1+4}.\Sp_4(2)\le \GL_4(3)$ does have the property that the proportion of $3$-elements in $G$ is at least $1/[4(3^4-1)]$, as mentioned above.

So we may assume that 
\begin{itemize}
    \item if $|\mathcal{X}|=1$, then $M_X= (d_X)_{p'}(q_1^{d_X}-1)/(q_1-1)$, where $\mathcal{X}=\{X\}$.
\end{itemize}
Now, notice that Rows 2--8 of Table \ref{tab:EXX} are all possibilities for $G_a/C_{G_a}(X)$ when $X=O_2(G_a)$, while Row 1 is a possibility for $G_a/C_{G_a}(X)$ when $X=O_3(G_a)$. It follows that:
\begin{itemize}
    \item if $X_i,X_j\in\mathcal{X}$ with $X_i\neq X_j$, then $G_a/C_{G_a}(X_i)$ and $G_a/C_{G_a}(X_j)$ cannot both lie in Rows 2--8 in Table \ref{tab:EXX}.
    \item if $X_i,X_j\in\mathcal{X}$ with $X_i\neq X_j$, and $G_a/C_{G_a}(X_i)\cong G_a/C_{G_a}(X_j)$ lie in the same row in Table \ref{tab:EXX}, then that row must be either Row 9, or Row 10.
\end{itemize}
It now follows from Lemma \ref{lem:strange} that $(q_1-1)\prod_{X\in\mathcal{X}}M_X\le (n_1)_{p'}(q_1^{n_1}-1)$, where $n_1:=\prod_{X\in\mathcal{X}}d_X$. Since $n_1$ divides $n/d$ by Lemma \ref{lem:MainPrimLin}, and
\begin{align}\label{lab:This}
M_p(G_a)\le (q_1-1)\left(\prod_{X\in\mathcal{X}}M_p(G_a/C_{G_a}(X))\right)\le (q_1-1)\prod_{X\in\mathcal{X}}M_X
\end{align}
by Corollary \ref{cor:PrimMpsCor}, the upper bound at (\ref{lab:MainClaim}), and hence the theorem, follows. \end{proof}

We conclude the section with the following corollary of the proof of Theorem \ref{thm:MainTheoremPrimLin}, which will be useful in our proof of Theorem \ref{thm:JordanAnalogue}.
\begin{cor}\label{cor:MTPL}
Let $q$ be a power of a prime $p$, and let $G\le\GL_n(q)$ be primitive. If $G$ is not one of those groups in Table \ref{tab:EXMainTheoremPrimLin}, and $(G,n,q)\neq (2^{1+2}.\Sp_4(2),4,3)$, then $M_p(G)\le n_{p'}(q^n-1)$. If $G$ is one of the groups in Table \ref{tab:EXMainTheoremPrimLin}, then $M_p(G)\le f(G,n,q)$, where $f(G,n,q)$ is as in the fourth column of the table. Finally, if $(G,n,q)= (2^{1+2}.\Sp_4(2),4,3)$, then $M_p(G)=576$. 
\end{cor}
\begin{proof}
The proof of Theorem \ref{thm:MainTheoremPrimLin} actually proves the statement in the corollary (by Corollary \ref{cor:MpCor}, proving that $M_p(G)\le M$ is stronger than proving that the proportion of elements of $p$-power order in $G$ is at least $1/M$). As can be seen therein, the primitive group $G=2^{1+4}.\Sp_4(2)\le \GL_4(3)$ has $M_p(G)=576$, and so does not satisfy $M_p(G)\le 4(3^4-1)$. (Since the proportion of $3$-elements in $G$ can be computed to be at least $1/[4(3^4-1)]$, this group is not an exception to Theorem \ref{thm:MainTheoremPrimLin}).
\end{proof}

\section{The proof of Theorems \ref{thm:JordanAnalogue} and \ref{thm:JordanIrr}}\label{sec:Jordanproof}
In this final section, we prove Theorems \ref{thm:JordanAnalogue} and \ref{thm:JordanIrr}. Before doing so, we require the following  easy lemma. Recall that the function $h_p:\mathbb{R}\rightarrow\mathbb{R}$ is defined as $h_p(m):=3^{(m-1)/2}$ if $p=2$; $h_p(m):=20^{(m-1)/4}$ if $p=3$; and $h_p(m):=(p-1)!^{(m-1)/(p-2)}$ if $p>3$. 
\begin{lemma}\label{lem:mtbound}
Let $m$ and $t$ be positive integers, and let $q$ be a power of a prime $p$. Then each of the following assertions hold. 
\begin{enumerate}[\upshape(i)]
        \item if $p=2$, then $[m_{2'}(q^m-1)]^th_p(t)\le [3(q^3-1)]^{mt/3}h_p(mt/3)$.
        \item if $q=3$, then $[m_{2'}(q^m-1)]^th_p(t)\le 640^{mt/4}h_p(mt/4)$.
        \item if $q=3$ and $4\nmid n$, then $[m(q^m-1)]^th_p(t)\le [2(q^2-1)]^{mt/2}h_p(mt/2)$.
        \item if $q\in\{9,27\}$, then $[m(q^m-1)]^th_p(t)\le [2(q^2-1)]^{mt/2}h_p(mt/2)$.
        \item if $q\in\{5,7,11\}$, then $[m(q^m-1)]^th_p(t)\le \alpha(q)^{mt/2}h_p(mt/2)$, where $\alpha(5):=96$, $\alpha(7):=144$, and $\alpha(11):=600$.
        \item if $q\in\{5,7,11\}$ and $n$ is odd, then $[m_{p'}(q^m-1)]^th_p(t)\le (q-1)^{mt}h_p(mt)$.
        \item if $p>2$ and $q\not\in\{3,5,7,9,11,27\}$, then $[m_{p'}(q^m-1)]^th_p(t)\le (q-1)^{mt}h_p(mt)$. Moreover, if $q\in\{13,19\}$ and $m(q)$ divides $n$, where $m(13):=4$ and $m(19):=2$, then $\alpha(q)^{n/m(q)}h_p(n/m(q))\le (q-1)^nh_p(n)$, where $\alpha(13):=138240$ and $\alpha(19):=1140$.
\end{enumerate}
\end{lemma}
\begin{proof}
The arguments to prove (i)--(vii) are entirely similar, so we will just prove (vi). So assume that $p>2$ and that $q\not\in\{3,5,7,9,11,27\}$. Clearly, we may also assume that $m\geq 2$. Define $k(p):=20^{1/4}$ if $p=3$, and $k(p):=(p-1)!^{1/(p-2)}$ if $p>3$. To prove the first part of (vi), it suffices to show that
\begin{align}\label{lab:bdbii}
    \left(\dfrac{q^m-1}{(q-1)^m}\right)^t\le \dfrac{k(p)^{t(m-1)}}{m^t}.
\end{align}
To this end, suppose first that $p>3$. Then since $q\geq 11$, we have $(q^m-1)/(q-1)^m\le q^m/(q-1)^m\le (11/10)^m$. Note also that $(11/10)^m\le 24^{(m-1)/3}/m$ for $m\geq 3$. Thus, since the function $n\rightarrow n!^{1/(n-1)}$ is non-decreasing for $n\in\mathbb{N}$, we have   
\begin{align*}
 \dfrac{q^m-1}{(q-1)^m}\le\left(\dfrac{5}{4}\right)^m\le \dfrac{24^{(m-1)/3}}{m}\le \dfrac{(p-1)!^{(m-1)/(p-2)}}{m}
\end{align*}
for $m\geq 2$. The bound at (\ref{lab:bdbii}) follows. The case $p=3$ and $q\geq 81$ is entirely similar.

Finally assume that $q\in\{13,19\}$, and let $m(q)$ and $\alpha(q)$ be as defined in (vi). Then one can check that $\alpha(q)^{1/m(q)}/(q-1)$ is less than $(p-1)!^{(1-1/m(q))/(p-2)}$ in each case, and this gives us what we need. 
\end{proof}

We are now ready to prove Theorems \ref{thm:JordanAnalogue} and \ref{thm:JordanIrr}. 
\begin{proof}[Proof of Theorems \ref{thm:JordanAnalogue} and \ref{thm:JordanIrr}] Let $G$ be a subgroup [respectively irreducible subgroup] of $\GL_n(q)=\GL(V)$. By Corollary \ref{cor:MpCor}, it will suffice to prove that 
\begin{align}\label{lab:NEED}
M_p(G)\le f(n,q)\text{ [resp. $i(n,q)$],}   
\end{align}
where $f(n,q)$ [resp. $i(n,q)$] is as defined in the statement of Theorem \ref{thm:JordanAnalogue} [resp. Theorem \ref{thm:JordanIrr}].

To this end, let $0=V_1<\hdots<V_d=V$ be a $G$-composition series for $V$. Then it is well-known that $G/O_p(G)$ is isomorphic to a completely reducible subgroup of $\GL(W)\cong \GL_n(q)$, where $W$ is the $\mathbb{F}_q[G]$-module $V_1\oplus V_2/V_1\oplus\hdots\oplus V_d/V_{d-1}$. Since $M_p(G)=M_p(G/O_p(G))$, we may therefore assume, for the remainder of the proof, that $G$ is a completely reducible subgroup of $\GL(V)$.

Suppose first that $G$ is irreducible. Then $G$ may be embedded as a subgroup of a wreath product $R\wr \Sym_t$, where $R\le \mathrm{GL}_m(q)$ is primitive, $mt=n$, and $G\cap R^t$ is a subdirect product in a direct product $A^t$ of $t$ copies of a normal subgroup $A$ of $R$ (we allow the case $t=1$, in which case $G\cong R$ is primitive). We now define a function $g$, as follows: if $R$ is not one of the groups from Theorem \ref{thm:MainTheoremPrimLin}(i), and $(R,m,q)$ is not $(2^{1+4}.\Sp_4(2),4,3)$, then define $g(R):=m_{p'}(q^m-1)$. If $R$ is one of the groups from Theorem \ref{thm:MainTheoremPrimLin}(i), then define $g(R):=P(G,m,q)$, where $P(G,m,q)$ is in the fourth column of Table \ref{tab:EXMainTheoremPrimLin}. Finally, if $(R,m,q)=(2^{1+4}.\Sp_4(2),4,3)$, then define $g(R):=M_3(R)=576$. Then by Corollary \ref{cor:MTPL}, $M_p(R)\le g(R)$ in each case. It then follows from Corollary \ref{cor:MpCor} that $M_p(G)\le g(R)^tM_p(G/G\cap R^t)\le g(R)^th_p(t)$, where the last inequality follows from Corollary \ref{cor:MainTheoremPermCor}. One can now deduce that $M_p(G)\le i(n,q)\le f(n,q)$ by inspecting the values of $g(R)$, and applying Lemma \ref{lem:mtbound} in each case.   

We can now complete the proof of (\ref{lab:NEED}) by induction on $n$. The irreducible case above can serve as the base case for induction, so assume that $G$ is reducible. Then $G$ embeds as a subgroup of $\GL_r(q)\times\GL_s(q)$, where $r,s\geq 1$, and $r+s=n$. Let $N:=G\cap (\GL_r(q)\times 1)$. Then the inductive hypothesis implies that $M_p(N)\le f(r,q)$ and $M_p(G/N)\le f(s,q)$. Since $M_p(G)\le M_p(N)M_p(G/N)$ by Corollary \ref{cor:MpCor}, and $f(r,q)f(s,q)=f(r+s,q)$, the claim (\ref{lab:NEED}) follows.
\end{proof}

\end{document}